\documentclass[10pt,a4wide]{amsart}

\usepackage{filecontents}

\usepackage{amsmath}
\usepackage{amssymb}
\usepackage{amsthm}
\usepackage{amsfonts, dsfont}
\usepackage{paralist}
\usepackage{graphics} 
\usepackage{epsfig} 
\usepackage{graphicx}  
\usepackage{epstopdf}
\usepackage{epstopdf}
\usepackage{verbatim}
\usepackage{mathrsfs}
\usepackage{mathtools}
\usepackage{pstricks}
\usepackage[colorlinks=true]{hyperref}
\hypersetup{urlcolor=blue, linkcolor=blue, citecolor=red}
\usepackage{cleveref}
\usepackage{relsize}
\usepackage{tikz}
\usetikzlibrary{matrix}
\usepackage{subcaption}
\usepackage{pgfplots}
\usepackage{fixltx2e}
\usepackage{enumitem}
\usepackage{upgreek}
\usepackage{url}
\usepackage{bm}
\usepackage{appendix}
\usepackage{lineno}

\newcommand{\D}[1]{\mbox{\rm #1}} 
\newcommand{\dd}{\D{d}}

\usepackage[abbrev]{amsrefs}

\usepackage{amssymb}

\numberwithin{equation}{section}

\theoremstyle{plain} 
\newtheorem{thm}{Theorem}[section]
\newtheorem{cor}{Corollary}[section]
\newtheorem{lem}{Lemma}[section]

\theoremstyle{defn}
\newtheorem{defn}{Definition}[section]

\theoremstyle{remark}
\newtheorem{rem}{Remark}[section]

\definecolor{ForestGreen}{RGB}{34,139,34}
\definecolor{ao(english)}{rgb}{0.0, 0.5, 0.0}

\begin{document}

\title[]{Relaxation and asymptotic expansion of controlled stiff differential equations}
\thanks{
}


\author{Michael Herty}
\address{Michael Herty \newline
\indent 
Institut f\"ur Geometrie und Praktische Mathematik\\RWTH Aachen University\\52062 Aachen\\Germany.
}
\email{\texttt{herty@igpm.rwth-aachen.de}}
\thanks{The authors thank the Deutsche Forschungsgemeinschaft (DFG, German Research Foundation) for the financial support through 320021702/GRK2326,  333849990/IRTG-2379, B04, B05 and B06 of 442047500/SFB1481, HE5386/19-3,23-1,25-1,26-1,27-1,30-1,  and  received funding from the European Union’s Horizon Europe research and innovation programme under the Marie Sklodowska-Curie Doctoral Network Datahyking (Grant No. 101072546).}

\author{Hicham Kouhkouh}
\address{Hicham Kouhkouh 
\newline
\indent Institut f\"ur Mathematik,  RTG ``Energy, Entropy, and Dissipative Dynamics'', RWTH Aachen University, 52062, Aachen, Germany 
}
\curraddr{
Department of Mathematics and Scientific Computing, NAWI, University of Graz, 8010, Graz, Austria 
}

\email{\texttt{kouhkouh@eddy.rwth-aachen.de}, \; \texttt{hicham.kouhkouh@uni-graz.at}}
\thanks{}





\date{\today}

\begin{abstract}
The control of relaxation-type systems of ordinary differential equations is investigated using the Hamilton--Jacobi--Bellman equation. First, we recast the model as a singularly perturbed dynamics which we embed in a family of controlled systems. Then we study this dynamics together with the value function of the associated optimal control problem.  We provide an asymptotic expansion in the relaxation parameter of the value function. We also show that its solution converges toward the solution of a  Hamilton--Jacobi--Bellman equation for a reduced control problem. Such systems are  motivated by semi-discretisation of kinetic and hyperbolic partial differential equations. Several examples are presented including Jin--Xin relaxation.
\end{abstract}


\maketitle

\section{Introduction}

We are interested in systems of ordinary differential equations (ODE) of the form
\begin{equation}
    \label{eq: sys intro 1}
    \dot{z}(s) = f(z(s),s) + \frac{1}{\varepsilon} \, g(z(s),s)
\end{equation}
where $\varepsilon>0$ is the stiffness parameter, representing for example a discretisation of a system of relaxation-type partial differential equations (PDE). 
Such systems have been studied intensively in the context of numerical methods, see e.g. \cites{MR2029975,MR2231946} for implicit--explicit discretisation methods, and in particular, in the PDE context where the above form arises, e.g. in semi-Lagrangian approximations to hyperbolic and kinetic transport equations \cites{MR3698447,MR2970736,MR1619910,MR3109810,MR3202241}.
Many approaches focus  on analytical \cite{MR1693210,MR3267352}  or numerical aspects  of the  previous relaxation--type system  \eqref{eq: sys intro 1}. To the best of our knowledge, only a few recent results exist dealing with such problem in the context of optimal control,  
see  \cites{MR3921027,MR4456458,MR4471483,MR2891921,MR2825372}. The latter publications focus on (high--order) numerical discretisation of an optimal control problem using Pontryagin's maximum principle. In particular, the  conditions on the numerical schemes have been a recent focus for discussion, see e.g. \cite{MR1804658,MR3054355,MR3072232}. 
Contrary to those approaches,  we rely on a formulation using the Hamilton--Jacobi--Bellman (HJB) equation. Our goal is twofold: first, we seek the limiting differential equation of \eqref{eq: sys intro 1} when $\varepsilon\to 0$, then we study this limit in the context of optimal control after embedding the latter in a family of parameterised differential equations. In particular, we provide an asymptotic expansion in $\varepsilon$ of the value function \eqref{eq: value intro 1} of such optimal control problem. Our analysis focuses on the value function of the optimal control problem and the corresponding HJB equation. It also allows for an asymptotic expansion of higher-orders, in particular when performed in the situation of Jin--Xin relaxation.

\subsection{Motivation}

The semi-discretisation of some partial differential equations of kinetic or hyperbolic type can exhibit stiffness yielding to an ODE of the form \eqref{eq: sys intro 1}. Its numerical implementation suffers from the lack of stability due to the smallness of the parameter $\varepsilon$ compared to the step size. Therefore, a challenging problem is the design of strategies that are reliable when dealing with such systems. Additionally, finding optimal controls which allow the dynamics to achieve a given goal while maintaining the stability of its numerical implementation is also highly desirable for example in the context of control of (nonlinear) PDEs. In particular when using an asymptotic expansion of the value function, and having in mind the control parameters expressed in their feedback form, one could quantify the impact of the stiffness in the optimal control. 
This would allow for more stability at the numerical level, and for a better design of optimal parameters. \\
\indent  
Although our main motivation is driven by  the numerical simulation of such equations, our goal is to prepare for a self-contained theoretical framework for which numerical simulations can later be undertaken and justified.

Let us mention that the difference between one of the author's thesis \cite{kouhkouh22phd} and the present manuscript lies in the particular structure of our model which enables us to construct a relaxed system that is much easier to work with, in particular in view of the applications we are motivated by. These applications have not been mentioned in the latter thesis, and hopefully pave the way for more results at their numerical level. Additionally, the three-scale analysis is also not present in \cite{kouhkouh22phd}. 

\subsection{The general approach}

Starting from \eqref{eq: sys intro 1}, we introduce a new variable $y(\cdot)$ whose dynamics captures the \textit{fast} part, that is
\begin{equation}
    \label{eq: fast intro 1}
    \dot{y}(s) = \frac{1}{\varepsilon} \, g(z(s),s). 
\end{equation}
Then, together with the dynamics of $z(\cdot)$, we obtain the following equivalent system of ordinary differential equations 
\begin{equation*}
    \begin{aligned}
        \dot{z}(s) & = f(z(s),s) + \dot{y}(s),\\
        \varepsilon\, \dot{y}(s) &= g(z(s),s),
    \end{aligned}
\end{equation*}
which takes the form 
\begin{equation}
    \label{eq: SP intro 1}
    \begin{aligned}
        \dot{\mathbf{z}}(s) & = F(\mathbf{z}(s), \mathbf{y}(s), s), \\
        \varepsilon\, \dot{\mathbf{y}}(s) &= G(\mathbf{z}(s), \mathbf{y}(s), s).
    \end{aligned}
\end{equation}
Indeed, we may set $\mathbf{z} := z-y$, $\mathbf{y} := y$, and then $F(\mathbf{z},\mathbf{y},s) := f(\mathbf{z}+\mathbf{y},s)$ and $G(\mathbf{z},\mathbf{y},s) := g(\mathbf{z}+\mathbf{y},s)$.

The formulation \eqref{eq: SP intro 1} has the advantage of being \textit{singularly perturbed} in time only. We refer to the monograph \cite{sanders2007averaging}. 
Such a structure may be easier to study, also when the dynamics are subject to an optimal control, and it benefits from a huge literature, e.g. the books \cite{bensoussan2011asymptotic, dontchev2006well, kokotovic1999singular}, the papers from dynamical systems viewpoint \cite{grammel2004nonlinear, grammel1997averaging, gaitsgory1999limit, artstein1997tracking, artstein2000value, gaitsgory1992suboptimization, gaitsgory2024averaging}, or from PDE viewpoint \cite{terrone2011limiting, alvarez2007multiscale, alvarez2008multiscale} to name but a few. We also refer to the thesis \cite{terrone2008singular, kouhkouh22phd} and the references therein. 

The next step is to embed the system \eqref{eq: SP intro 1} into a family of parameterised (controlled) system of ODEs
\begin{equation}
    \label{eq: CSP intro 1}
    \begin{aligned}
        \dot{\mathbf{z}}(s) & = F(\mathbf{z}(s), \mathbf{y}(s), \alpha(s), s), \quad \mathbf{z}(0) = \mathbf{z}_0, \\
        \varepsilon \, \dot{\mathbf{y}}(s) &= G(\mathbf{z}(s), \mathbf{y}(s), \beta(s), s), \quad \mathbf{y}(0)=\mathbf{y}_0.
    \end{aligned}
\end{equation}
This dynamics now includes two control parameters $\alpha(\cdot)$ and $\beta(\cdot)$, that are chosen in order to minimise the following cost functional 
\begin{equation}\label{eq: cost}
    J(\mathbf{z}) = \int_{0}^{t}\ell(\mathbf{z}(s))\,\dd s + \phi(\mathbf{z}(t)).
\end{equation}
The precise assumptions on the data of the problem will be made precise soon after (see Section \ref{sec: hjb}). In particular, our analysis will be focused on systems that are affine in the control. 

The analysis of control problems when viewed from the PDE perspective relies on the celebrated HJB equation whose solution is the value function defined by
\begin{equation}
    \label{eq: value intro 1}
    V^{\varepsilon}(\mathbf{z}_0, \mathbf{y}_0, t) = \inf\limits_{ (\alpha,\beta) } \; J(\mathbf{z}), \quad \text{ s.t.: } \quad \eqref{eq: CSP intro 1}.
\end{equation}

In the present manuscript, we are interested in the limit of $V^\varepsilon$ when $\varepsilon \to 0$, as well as in  an asymptotic expansion of the form
\begin{equation*}
     V^{\varepsilon} = V_{0} + \varepsilon\, V_{1}  + O(\varepsilon^2). 
\end{equation*}
This ansatz is inspired by the earlier work in \cite{o1974introduction, lions2006perturbations, bensoussan2011asymptotic}, provided one could justify existence of each one of the terms in the right-hand side. 
Such expansion is motivated by the design of optimal controls which are usually obtained from (the gradient of) the value function when they are considered in the feedback form.

In the second part, we expand the analysis and consider in equation \eqref{eq: sys intro 1} a function $g(z,t) = g^{\varepsilon}(z,t)$ depending on $\varepsilon$ such that
\begin{equation}
    \label{eq: g eps}
    g^{\varepsilon}(z,t) = g_{0}(z,t) + \varepsilon\, g_{1}(z,t) + O(\varepsilon^{2}).
\end{equation}
The goal thereafter will be to analyse the impact of  $g_{0}$
 and $g_{1}$ on the asymptotic expansion of the value function defined in \eqref{eq: value intro 1}.

The paper is organised as follows. We study a system of the form \eqref{eq: CSP intro 1} that is affine in the control parameters. In \textbf{Section \ref{sec: linear}} we provide a first analysis of the latter and the convergence of such trajectories as $\varepsilon\to 0$. In \textbf{Section \ref{sec: hjb}}, we focus on the value function of the corresponding control problem. We start from an ansatz on its first order asymptotic expansion as $\varepsilon\to 0$, then we show the existence of each one of the terms in this expansion. We also provide a control representation of the term in the leading order. This is summarised in \textbf{Theorem \ref{thm conv}}.  
Our next main contribution is in the higher-order asymptotic expansion which is the object of \textbf{Section \ref{sec: multiscale sys}}. We will show in \textbf{Theorem \ref{thm: conv multi}} how to obtain the terms in the second-order expansion, but the same method can be analogously be performed for higher order terms. This is then applied in \textbf{Section \ref{Sec: Application}} to the celebrated Jin--Xin relaxation scheme for which we obtain \textbf{Corollary \ref{cor: JX}} and \textbf{Corollary \ref{cor JX 3}}. We conclude the manuscript with \textbf{Appendix \ref{appendix}} where we state the abstract optimal control problem obtained when $\varepsilon\to 0$ in a more general setting. Lastly, we show in \textbf{Appendix \ref{appendix examples}} how to apply our result with several examples.

\section{Singular perturbations for affine optimal control problems}\label{sec: linear}

Denote by $z$ and $y$  the slow and fast variables respectively, subject to the following system of singularly perturbed ODEs 
with affine control inputs
\begin{equation}
    \label{CSP}
    \begin{aligned}
        \dot{z}(s) & = A_{1}(z)y + B_{1}(z)\alpha + C_{1}(z),& z(0) = z_{0}\in \mathds{R}^{m},&\\ 
        \varepsilon\dot{y}(s) & = A_{2}(z)y + B_{2}(z)\beta + C_{2}(z),& y(0) = y_{0}\in\mathds{R}^{n},& 
    \end{aligned}
\end{equation}
where the time variable $s\in[0,1]$, and $\varepsilon>0$ is a small parameter. In this section, we will assume the following. \\
\indent \textbf{Assumptions and notations (A): } \vspace*{-3mm}
\begin{enumerate}[label = \textbf{(a.\arabic*)}]
    \item $z(\cdot):[0,1] \to \mathds{R}^{m}$ and $y(\cdot):[0,1]\to \mathds{R}^{n}$ are the state trajectories. 
    \item $\alpha(\cdot):[0,1]\to \Omega_{A}$ and $\beta(\cdot):[0,1]\to \Omega_{B}$ are measurable functions, and are the admissible controls. 
    \item The control sets $\Omega_{A}\subset \mathds{R}^{p}$ and $\Omega_{B}\subset \mathds{R}^{q}$ are both compact and convex. 
    \item $A_{1}(\cdot): \mathds{R}^{m}\to \mathds{M}^{m,n}$ is a matrix-valued function.
    \item $A_{2}(\cdot): \mathds{R}^{m}\to \mathds{M}^{m,m}$ is a matrix-valued function.
    \item $\exists\,\lambda>0$ such that the matrix function $A_{2}(\cdot)$ enjoys the stability condition
    \begin{equation}\label{assumption fast}
    \left\| e^{A_{2}(z)t} \right\| \leq e^{-\lambda t},\; \text{locally uniformly in } z, 
    \end{equation}
    i.e. the eigenvalues of the matrix $A_{2}$ have negative real parts.
    \item $B_{1}(\cdot):\mathds{R}^{m}\to \mathds{M}^{m,p}$ and $B_{2}(\cdot):\mathds{R}^{m}\to \mathds{M}^{n,q}$ are matrix-valued functions. 
    \item $C_{1}(\cdot):\mathds{R}^{m}\to \mathds{R}^{m}$ and $C_{2}(\cdot):\mathds{R}^{m}\to \mathds{R}^{n}$ are vector-valued functions. 
    \item $\Phi(\cdot):\mathds{R}^{n}\to \mathds{R}$ is continuous, and shall be the cost functional. 
\end{enumerate}

We will also use the notation $p\cdot q = p^{\top}q = \sum_{i}p_{i}q_{i}$ to denote the scalar product of two vectors $p,q$ of the same size. Here $p^{\top}$ is the transpose of the vector $p$. 

The control parameters $\alpha,\beta$ are chosen in order to achieve the minimisation 
\begin{equation}\label{SP}
\tag{SP}
    G_{\varepsilon}:=\inf \; \Phi(z(1)),
\end{equation}
subject to the singularly perturbed and controlled trajectories \eqref{CSP} with $s\in [0,1]$.

Next, we set $\varepsilon=0$ in \eqref{CSP} and solve the static equation:
\begin{equation}
\label{static}
    0 = A_{2}(z)y + B_{2}(z)\beta + C_{2}(z),
\end{equation}
whose root $y=\psi(z,\beta)$ is then substituted in the first equation in \eqref{CSP}, yielding what we shall refer to as the reduced dynamics
\begin{equation}
    \label{reduced}
    \dot{\bar{z}}(s) = A_{1}(\bar{z})\psi(\bar{z},\beta) + B_{1}(\bar{z})\alpha + C_{1}(\bar{z}), \quad \bar{z}(0) = z_{0}\in \mathds{R}^{m}.
\end{equation}
This controlled ODE is associated with the minimisation of the cost functional
\begin{equation}\label{R}
\tag{R}
    \bar{G}:=\inf \; \Phi(\bar{z}(1)).
\end{equation}

We are now ready to state some results and definitions from \cite{gaitsgory1992suboptimization} on the limiting problem when $\varepsilon\to 0$. 
\begin{defn}
Given a trajectory $x(\cdot)$, an admissible control $(\alpha_{x(\cdot)},\beta_{x(\cdot)})$ is referred to as \underline{$x(\cdot)$-approximating} if  it generates in \eqref{CSP} a trajectory $z_{\varepsilon}(\cdot)$ satisfying 
\begin{equation*}
    \max\limits_{s\in[0,1]} \|z_{\varepsilon}(s) - x(s)\| \leq \mu(\varepsilon) \to 0 \text{ as } \varepsilon\to 0.
\end{equation*}
\end{defn}
Note that in this definition, the control $(\alpha_{x(\cdot)},\beta_{x(\cdot)})$ is not necessarily the control of the admissible trajectory $x(\cdot)$. Moreover, $\alpha$ and $\beta$ do not need to be necessarily different.

\begin{defn}\label{def:approx}
    We say that the problem \eqref{SP} is \underline{approximated by} the problem \eqref{R} if $G_{\varepsilon}\to \bar{G}$ as $\varepsilon\to 0$ and if, corresponding to any admissible trajectory $\bar{z}_{\nu}(t)$ of the reduced system \eqref{reduced} such that $\Phi(\bar{z}_{\nu}(1))\leq \bar{G}+\nu$, $\nu>0$, there exists $\bar{z}_{\nu}(\cdot)$-approximating control providing in the problem \eqref{SP} the value of the functional differing from the optimal one by $\nu+\kappa(\varepsilon)$, where $\kappa(\varepsilon)\to 0$ as $\varepsilon\to 0$.
\end{defn}

\begin{thm}\label{thm: conv G}
Under the assumptions (A), the problem \eqref{SP} is approximated by the problem \eqref{R}.
\end{thm}
\noindent For the proof we refer to Theorem 4.3 and Example 4.3 in \cite{gaitsgory1992suboptimization}.

\begin{rem}\label{rem periodic}
Assumption \eqref{assumption fast} guarantees that all admissible trajectories $\{z(s),y(s)\}$ of \eqref{CSP} remain in a compact subset of $\mathds{R}^{m}\times \mathds{R}^{n}$ for all $s\in [0,1]$ and for all initial positions $(z_{0}, y_{0})$ chosen in a compact subset of $\mathds{R}^{m}\times \mathds{R}^{n}$. We can therefore assume without loss of generality that $y(\cdot)$ remains in the $n$-dimensional torus $\mathds{T}^{n}$, and the dynamics \eqref{CSP} is periodic in the $y$-variable.
\end{rem}

The convergence stated in Theorem \ref{thm: conv G} is further investigated in the next section, seeking how it translates at the level of the HJB equation and the value function. 

\section{On the HJB equation and the reduction method}\label{sec: hjb}

In this section, we analyse how the reduction \eqref{reduced} extends to  the
HJB equation associated with the control problems. Consider a finite horizon control problem whose cost functional  is
\begin{equation}
    \label{cost}
    J(z) := \int_{0}^{t}  \ell(z(s)) \,\dd s + \phi(z(t)),
\end{equation}
where  $\ell,\phi$ are real continuous functions. The value function is 
\begin{equation}\label{value eps}
    V^{\varepsilon}(z,y,t) := \inf\limits_{\alpha(\cdot),\beta(\cdot)}\, J(z), \quad \text{ s.t.: } \; \eqref{CSP}\,,
\end{equation}
and $\alpha(\cdot),\beta(\cdot) \in L^{\infty}([0,+\infty))$ with values in the compact and convex sets $\Omega_A$ and $\Omega_B$ respectively, subsets of $\mathds{R}^{p}$ and $\mathds{R}^{q}$.

\begin{rem}
It is well-known that such a control problem can be reformulated as the problem \eqref{SP} which is of Mayer type. To do so, it suffices to add a new \textit{slow} variable $x(\cdot)\in \mathds{R}$, and the system becomes 
\begin{equation*}
    \begin{aligned}
        \dot{x}(s) & = \quad \; 0\quad \, + \quad \,  0 \quad \;\, +\quad \ell(z), & \;x(0) = 0\in \mathds{R},\, \; \quad&\\
        \dot{z}(s) & = A_{1}(z)y + B_{1}(z)\alpha + C_{1}(z),& z(0) = z_{0}\in \mathds{R}^{m},&\\ 
        \varepsilon\dot{y}(s) & = A_{2}(z)y + B_{2}(z)\beta + C_{2}(z),& y(0) = y_{0}\in\mathds{R}^{n}.&
    \end{aligned}
\end{equation*}
This is of the form \eqref{CSP} if we consider instead of $z$, the new slow variable $(x,z)$. And while all the functions depend on $z$, they could be indeed functions of the whole variable $(x,z)$. 
Then the objective function would be to minimise: $\inf\; x(t) + \phi(z(t))$. 
The time horizon being finite, it can easily be parameterised to $s \in [0,1]$, and $t=1$.
\end{rem}

We shall now state our standing assumptions (SA) that we assume to hold in the rest of the paper.\\ 
\indent \textbf{The standing assumptions (SA):}\vspace*{-2mm}
\begin{enumerate}[label = \textbf{(SA.\arabic*)}]
    \item The functions $\ell(\cdot), \phi(\cdot)$ are bounded and uniformly continuous.
    \item The state trajectories $z(\cdot),y(\cdot)$, and the controls $\alpha,\beta$ are defined in \textbf{(a.1,2)}.
    \item The dynamics $y(\cdot)$ lives in an $n$-dimensional torus $\mathds{T}^{n}$, and all the functions are periodic in $y$; see Remark \ref{rem periodic}.
    \item The functions $A_{i}(\cdot), B_{i}(\cdot), C_{i}(\cdot), i=1,2$ are defined in the assumptions \textbf{(a.4,5,7,8)}, and are moreover locally Lipschitz with at most a linear growth.
    \item The function $A_{2}(\cdot)$ enjoys the stability condition \eqref{assumption fast} in \textbf{(a.6)}. 
    \item For all $z\in \mathds{R}^{m}$, the matrix $B_{2}(z)$ has full rank. 
    \item The control sets $\Omega_{A}$ and $\Omega_{B}$ are sufficiently large compact and convex. 
\end{enumerate} 

The regularity assumptions are classical assumptions to ensure well-posedeness of the ODEs and for the study of the optimal control problem. 
Assumption (SA.5) ensures the validity of our Remark \ref{rem periodic}, hence justifying the periodicity assumption in (SA.3). This is needed to gain compactness in the singularly perturbing variables (the fast ones) and pass to the limit. The unbounded case is more delicate (see \cite{bardi2023singular, bardi2022deep}). 
The last two assumptions, together with Remark \ref{rem periodic} ensure the following controllability condition is satisfied
\begin{equation}\label{strong controllability}
    \begin{aligned}
         & \text{for any } z \text{ fixed}, \text{ there exists } \, r(z)>0 \, \text{ such that: } \\
         & B\big(0,r(z)\big) \subseteq \overline{\text{co}}\big\{ A_{2}(z)y + B_{2}(z)\beta + C_{2}(z)\, : \beta \in \Omega_B\big\}, \; \forall y\in \mathds{T}^{n},
    \end{aligned}
\end{equation} 
where $B\big(0,r(z)\big)$ is the ball centred in $0$ with radius $r(z)$, and $\overline{\text{co}}$ refers to the closed convex hull. 
Indeed, as the dynamics $y$ remains in a compact set and the matrix $B_2$ has full rank, the control can span a set that is large enough to contain a ball around $0$. This means that with this control, one could generate all directions of the state space having then access to all positions, which ensures the controllability of the system. This condition is needed for Lemma \ref{lem: corrector} and Theorem \ref{thm conv}.

By standard results on viscosity solutions for Hamilton-Jacobi-Bellman (HJB) equation (see e.g. \cite[Chapter III]{bardi2008optimal}), the value function is such that $V^{\varepsilon}\in C(\mathds{R}^{m}\times\mathds{R}^{n}\times [0,T])$ for $T<\infty$, and it is the unique viscosity solution to the HJB equation
\begin{equation}\label{HJB}
\left\{ \; 
\begin{aligned}
    & \partial_{t} V^{\varepsilon} + H\left(z,y,D_{z}V^{\varepsilon},\frac{1}{\varepsilon}D_{y}V^{\varepsilon}\right) = \ell(z),\quad \text{in } \mathds{R}^{m}\times\mathds{R}^{n}\times (0,T),\\
    & V^{\varepsilon}(z,y,0) = \phi(z), \text{ in } \mathds{R}^{m}\times\mathds{R}^{n},
\end{aligned}
\right.
\end{equation}
where the Hamiltonian $H$ is
\begingroup
\allowdisplaybreaks
    \begin{align*}
    & H(z,y,p,q)\\
    & = \sup\limits_{\alpha\in \Omega_{A}, \beta\in \Omega_{B}}\{-p\cdot (A_{1}(z)y + B_{1}(z)\alpha + C_{1}(z)) - q\cdot (A_{2}(z)y + B_{2}(z)\beta + C_{2}(z))\}\\
    & = -(A_{1}(z)y + C_{1}(z))\cdot p + \sup\limits_{\alpha\in \Omega_A}\big\{ -p\cdot B_{1}(z)\alpha  \big \} \\
    & \quad \quad \quad - (A_{2}(z)y  + C_{2}(z))\cdot q  + \sup\limits_{\beta\in \Omega_B}\big\{ -q\cdot B_{2}(z)\beta \big \}.
    \end{align*}
\endgroup

In the sequel, and only for the sake of simplicity of notation, we shall drop the dependence of $A_{i},B_{i},C_{i},i=1,2$ on the variable $z$. 

\subsection{The first--order asymptotic expansion} \label{sec: first order asymptot}

Given the optimal control problem \eqref{value eps}, we are interested in the limit as $\varepsilon\to 0$ of both the value function $V^{\varepsilon}$ (at the level of the PDE problem \eqref{HJB}) and the underlying controlled and singularly perturbed trajectories \eqref{CSP}. In \cite{homocp} (see also Chapter I in \cite{kouhkouh22phd}), this convergence is proved using \textit{Limit Occupational Measure Sets} (LOMS); see Appendix \ref{appendix}. These measures are essential in the construction of the limiting optimal control problem and they have been  used already in \cite{gaitsgory1992suboptimization} and references therein. 
\par 
Using the results in \cite{homocp} (see also \cite[Chapter I]{kouhkouh22phd}), we show in Theorem \ref{thm conv} below that the solution $V^{\varepsilon}$ to \eqref{HJB} converges locally uniformly to $V_{0}$ unique continuous viscosity solution in $\mathds{R}^{m}\times (0,T)$ of
\begin{equation}\label{HJB 0}
\left\{\;
\begin{aligned}
    & \partial_{t} V_{0}  - C_{1}(z)\cdot\partial_{z}V_{0} + \sup\limits_{\alpha\in \Omega_A}\{-\partial_{z}V_{0}\cdot B_{1}(z)\alpha  \} + \bar{\lambda}_{1}(z,\partial_{z}V_{0}) = \ell(z), \\
    & V_{0}(z,0) = \phi(z), \text{ in } \mathds{R}^{m},
\end{aligned}
\right.
\end{equation}
where, for fixed $z$ and $p:=\partial_{z}V_{0}(z)$, we have $\bar{\lambda}_{1}(z,p)$ a constant in $y$ for which there exists a continuous viscosity solution $V_{1}(y)$ to the so-called \textit{cell problem}
\begin{equation}\label{cell}
    \begin{aligned}
          - (A_{2}y + C_{2})\cdot \partial_{y}V_{1}+\sup\limits_{\beta\in \Omega_B}\{ \; -\partial_{y}V_{1}\cdot B_{2}\beta  \} - p^{\top}A_{1}y = \bar{\lambda}_{1}(z,p).
    \end{aligned}
\end{equation}

In this case, one would have the following first order asymptotic expansion 
\begin{equation}\label{expansion 1}
    V^{\varepsilon}(t,z,y) = V_{0}(t,z) + \varepsilon V_{1}(y) + O(\varepsilon^{2}),
\end{equation}
where $V_{0}$ solves \eqref{HJB 0}.

More precisely, given the optimal control problem \eqref{value eps} where the dynamics $(y(\cdot), z(\cdot))$ in \eqref{CSP} satisfies the standing assumptions (SA), 
we have the following.

\begin{lem}\label{lem: corrector}
    Let $z,p$ be fixed in $\mathds{R}^{m}$. Under the standing assumptions (SA), there exists a unique $\bar{\lambda}_{1}(z,p)\in \mathds{R}$ such that \eqref{cell} has a continuous periodic viscosity solution $V_{1}(\cdot)$.
\end{lem}

\begin{proof}
    This is a particular case of \cite[Proposition 6.2]{alvarez2010ergodicity} where it is assumed that the fast dynamics in our case is \textit{small time controllable}. But this is ensured by the controllability assumption \eqref{strong controllability}, see \cite[Example 1 in \S 6.1]{alvarez2010ergodicity}. 
\end{proof}

\begin{thm}\label{thm conv}
    As $\varepsilon\to 0$, the sequence $V^{\varepsilon}$ of solutions of \eqref{HJB} converges locally uniformly on $\mathds{R}^{m}\times (0,+\infty)$ to $V_{0}$ solution to
    \begin{equation}
        \label{HJB eff}
        \left\{\,
        \begin{aligned}
            & \partial_{t} V_{0} + \bar{H}\left(z,D_{z}V_{0}\right) = \ell(z),\quad \text{in } \mathds{R}^{m}\times (0,T),\\
            & V_{0}(z,0) = \phi(z), \text{ in } \mathds{R}^{m},
        \end{aligned}
        \right.
    \end{equation}
    and the limit (effective) Hamiltonian is
    \begin{equation*}
    \begin{aligned}
        \bar{H}(z,p) &= \sup\limits_{\alpha\in \Omega_{A},\beta\in \Omega_B}\left\{-p\cdot (A_{1}(z)\psi(z,\beta) + B_{1}(z)\alpha + C_{1}(z)) \right\},
    \end{aligned}
    \end{equation*}
    where $\psi(z,\beta)$ solves (for $y$) the static equation \eqref{static}. Moreover $V_{0}$ is the value function of the limit (effective) optimal control problem
    \begin{equation}\label{value eff}
    V_{0}(z,t) := \inf\limits_{\alpha(\cdot),\beta(\cdot)}\, J(z), \quad \text{ s.t.: } \; \eqref{reduced}\,,
    \end{equation}
    where $J$ is as defined in \eqref{cost}.
\end{thm} 

This theorem shows the convergence for $\varepsilon\to 0$ of the solution to the HJB. Applying the previous result and with the notation in \eqref{HJB 0}, we have 
\begin{equation*}
\begin{aligned}
    &\bar{H}(z,p) = -p\cdot C_{1}(z) + \sup\limits_{\alpha\in \Omega_A} \{-p\cdot B_{1}(z)\alpha \} + \bar{\lambda}_{1}(z,p),\\ 
    &\text{and } \quad \bar{\lambda}_{1}(z,p) = \sup\limits_{\beta\in \Omega_B} \{-p \cdot A_{1}(s)\psi(z,\beta) \}.
\end{aligned}
\end{equation*}
Furthermore, using \eqref{expansion 1}, one could obtain an asymptotic expansion of a closed-loop optimal control.

\begin{proof}
The standing assumptions (SA) guarantee the controllability condition \eqref{strong controllability}. We can then apply Theorem \cite[Theorem 3.1]{homocp} (or \cite[Theorem 1.3.1]{kouhkouh22phd})which ensures the value function $V^{\varepsilon}$ converges locally uniformly on $\mathds{R}^{m}\times(0,+\infty)$ to $\widetilde{V}$ solution to 
\begin{equation}
    \left\{ \; 
    \begin{aligned}
        & \partial_{t}\widetilde{V} + \widetilde{H}(z,D_{z}\widetilde{V}) = \ell(z),\; \text{ in } \mathds{R}^{m}\times(0,T),\\
        & \widetilde{V}(z,0) = \phi(z),\; \text{ in } \mathds{R}^{m},
    \end{aligned}
    \right.
\end{equation}
where 
\begin{equation*}
    \begin{aligned}
        \widetilde{H}(z,p) &= \sup\limits_{\alpha\in \Omega_{A},\mu\in \mathfrak{L}(z)}
        -p\cdot \int_{\mathds{T}^{n}\times \Omega_{B}}(A_{1}(z)y + B_{1}(z)\alpha + C_{1}(z)) \, \text{d}\mu(y,\beta),
    \end{aligned}
\end{equation*}
and $\mathfrak{L}(z)$ is the \textit{Limit Occupational Measure Set} (LOMS; see the definition in the Appendix \ref{appendix}). It turns out, however, that in the setting  of section \S \ref{sec: linear}, the LOMS coincides with the reduced dynamics as was proven in \cite[Example 3.4]{gaitsgory1992suboptimization}.  Indeed, let $\bar{\beta}\in \Omega_B$ be arbitrarily fixed, then using \cite[eq. (3.21)]{gaitsgory1992suboptimization}, the dynamics \eqref{CSP} at the limit $\varepsilon\to 0$ becomes 
\begin{equation*}
    \dot{\tilde{z}} = A_{1}(\tilde{z})\psi(\tilde{z},\bar{\beta}) + B_{1}(\tilde{z}) \alpha + C_{1}(\tilde{z}),
\end{equation*}
where $\psi(\tilde{z},\bar{\beta})$ solves \eqref{static}, that is
\begin{equation*}
    \psi(\tilde{z},\bar{\beta}) = - A_{2}^{-1}B_{2}\bar{\beta} - A_{2}^{-1}C_{2},
\end{equation*}
and is unique (see \cite[Lemma 3.1 (ii)]{gaitsgory1992suboptimization}). This leads to the reduced dynamics \eqref{reduced}, which can be equivalently represented as 
\begin{equation*}
    \dot{\tilde{z}} = \int_{\mathds{T}^{n}\times \Omega_B} \big[ A_{1}(\tilde{z})y + B_{1}(\tilde{z})\alpha  + C_{1}(\tilde{z})\,\big]\, \delta_{\psi(\tilde{z},\beta)}(\dd y) \otimes \delta_{\bar{\beta}}(\dd \beta).
\end{equation*}
Therefore, any $\mu \in \mathfrak{L}(z)$ corresponds to $\delta_{\psi(\tilde{z},\bar{\beta})}(\text{d}y)\!\otimes \!\delta_{\bar{\beta}}(\text{d}\beta)$, where $\psi(z,\bar{\beta})$ solves \eqref{static} for some $\bar{\beta}\in \Omega_B$, and so the supremum in the definition of $\widetilde{H}$ is taken over $\alpha\in \Omega_A$ and $\bar{\beta}\in \Omega_B$ which ultimately yields to $\widetilde{H}=\Bar{H}$. 
\end{proof}

\subsection{A tentative approach for second--order approximation}

The next term in the asymptotic expansion can be formally obtained by making the following 

\noindent \textbf{Ansatz. } $V^{\varepsilon}(z,y,t) = V_{0}(z,t;\varepsilon) + \varepsilon\,V_{1}(y;\varepsilon) + \varepsilon^{2}\,V_{2}(y) + O(\varepsilon^{3}).$

Using this ansatz in the PDE in \eqref{HJB} yields
\begin{equation*}
    \begin{aligned}
        & 0= \partial_{t} V_{0}  - C_{1}\cdot\partial_{z}V_{0} + \sup\limits_{\alpha\in \Omega_A}\{-\partial_{z}V_{0}\cdot B_{1}\alpha \} -\ell(z) - A_{1}y\cdot\partial_{z}V_{0}  \\
        & \quad \quad - (A_{2}y + C_{2})\cdot\partial_{y}V_{1} + \sup\limits_{\beta\in \Omega_B}\{ \; -\partial_{y}V_{1}\cdot B_{1}\beta  -\varepsilon(A_{2}y + C_{2} + B_{2}\beta)\cdot\partial_{y}V_{2}  \; \}.
    \end{aligned}
\end{equation*}
We identify the terms depending on the fast variable $y$ according to the powers of $\varepsilon$. This implies  having  pairs $(\lambda_{1},V_{1}(\cdot))$ and $(\lambda_{2}, V_{2}(\cdot))$ where $\lambda_{1},\lambda_{2}$ are constant in $y$ and $V_{1},V_{2}$ are functions of $y$ for fixed $z$, $\partial_{z}V_{0}=:p$ and $\beta$,   such that
\begin{itemize}
    \item The PDE problem solved by $(\lambda_{2}, V_{2}(\cdot))$ is
    \begin{equation}\label{HJB 2}
    -(A_{2}y + C_{2} + B_{2}\beta)\cdot\partial_{y}V_{2}  = \lambda_{2}(z,\beta), \quad \text{for } y\in \mathds{T}^{n}.
\end{equation}
    \item The PDE problem solved by $(\lambda_{1}, V_{1}(\cdot))$ is
    \begin{equation}\label{HJB 1}
    \begin{aligned}
        & - A_{1} y \cdot p - (A_{2}y + C_{2})\cdot\partial_{y}V_{1} \\
        & \quad \quad \quad + \sup\limits_{\beta\in \Omega_B}\{ \; -\partial_{y}V_{1}\cdot B_{1}\beta  +\varepsilon \lambda_{2}(z,\beta)  \} = \lambda_{1}(z,p,\varepsilon), \quad \text{for } y\in \mathds{T}^{n}.
    \end{aligned}
    \end{equation}
\end{itemize}
Then the PDE in \eqref{HJB}, complemented with the same initial condition, becomes
\begin{equation}\label{HJB eps}
    \partial_{t} V_{0}  - C_{1}(z)\cdot\partial_{z}V_{0} + \sup\limits_{\alpha\in \Omega_A}\{-\partial_{z}V_{0}\cdot B_{1}(z)\alpha  \} + \lambda_{1}(z,\partial_{z}V_{0},\varepsilon) = \ell(z).
\end{equation}

The results of the previous section are not straightforwardly applicable in this situation. This is mainly due to the presence of $\beta$ in both PDEs \eqref{HJB 2} and \eqref{HJB 1}. Indeed, in order to obtain an asymptotic expansion of $V^{\varepsilon}$, one could start first by solving \eqref{HJB 2} for $\beta$ fixed, then plugging $\lambda_{2}$ in \eqref{HJB 1} and solving the latter PDE (provided $\lambda_{2}$ is continuous in $\beta$). The next step would be to find $\beta^{*}$ for which the $\sup$ is obtained and solve again \eqref{HJB 2} with $\beta^{*}$. Then the resulting function $V_{2}$ is the next term in the asymptotic expansion
\begin{equation}\label{expansion 2}
    V^{\varepsilon}(t,z,y) = {V}_{0}(t,z;\varepsilon) + \varepsilon {V}_{1}(y;\varepsilon) + \varepsilon^{2}V_{2}(y) + O(\varepsilon^{3}),
\end{equation}
where $V_{0}$ solves \eqref{HJB eps} and $V_{1}$ solves \eqref{HJB 1}. \\
\indent To that must be added the other difficulty arising from the ambiguity of the dependence of $V_{0}(t,z;\varepsilon), V_{1}(z;\varepsilon)$ on $\varepsilon$ (see \eqref{HJB 1}-\eqref{HJB eps}). Note that even if there was no control $\beta$ in the slow variables, one would have $\lambda_{1}(z,p,\varepsilon)= \bar{\lambda}_{1}(z,p) + \varepsilon\lambda_{2}(z)$ where $\bar{\lambda}_{1}(z,p)$ is the constant obtained in \eqref{cell}, and $V_{0}, V_{1}$ will still depend on $\varepsilon$ in an unknown fashion.

To bypass these difficulties (mainly, the presence of $\beta$ in \eqref{HJB 2} and the dependence on $\varepsilon$ in \eqref{HJB 1}), and still keep track of $\varepsilon$ in the  asymptotic expansion, we will consider a multi-scale approach. This is the object of the next section.

\section{A Multi-Scale System}\label{sec: multiscale sys}

 \subsection{The three-scale Optimal control problem}\label{sec: linear 2}

The state variables $z,y$ and $x$ will refer to the macro-, meso-, and micro-scale regime respectively, subject to the following system of singularly perturbed and controlled ODEs 
\begin{equation}
    \label{CSP 2}
    \begin{aligned}
        \dot{z}(s) & = A_{0}(z)x + A_{1}(z)y + B_{1}(z)\alpha + C_{1}(z),& z(0) = z_{0}\in \mathds{R}^{m},&\\ 
        \varepsilon\dot{y}(s) & = A_{2}(z)y + B_{2}(z)\beta + C_{2}(z),& y(0) = y_{0}\in\mathds{R}^{n},&\\ 
        \varepsilon^{2}\dot{x}(s) & = A_{3}(z)x + B_{3}(z)\gamma + C_{3}(z),& x(0) = x_{0}\in\mathds{R}^{\ell},&
    \end{aligned}
\end{equation}
where $s\in[0,1]$, $\varepsilon>0$ is a small parameter, $A_{0}(z), A_{i}(z), B_{i}(z), C_{i}(z), i=1,2,3$ are functions
satisfying the standing assumptions \textbf{(SA)} in Section \ref{sec: hjb}, moreover $x(\cdot)$ satisfy the stability and controllability assumptions (as it is assumed for $y(\cdot)$) that is: the matrix $A_{3}(\cdot)$ has all its eigenvalues negative, and $B_{3}(\cdot)$ has full rank. We recall that $y,x$ live on the torus $\mathds{T}^{n}$ and $\mathds{T}^{\ell}$ respectively and that all functions if depending on $y$ or $x$ are periodic in the latter (see Remark \ref{rem periodic}). 

The admissible controls $\alpha(\cdot)$, $\beta(\cdot)$ and $\gamma(\cdot)$ are measurable functions with values in $\Omega_A$, $\Omega_B$ and $\Omega_\Gamma$ respectively, all compact and convex subsets of $\mathds{R}^{p}$, $\mathds{R}^{q}$ and $\mathds{R}^{r}$ respectively. They are chosen  as to minimise the cost functional
\begin{equation}\label{SP 2}
\tag{SP}
    \inf \; \int_{0}^{1} \ell(z(s))\,\text{d}s +  \phi(z(1)),
\end{equation}
subject to the singularly perturbed and controlled trajectories \eqref{CSP 2} with $s\in [0,1]$.

The corresponding HJB equation takes the form
\begin{equation}\label{HJB multi}
\left\{ \; 
\begin{aligned}
    & \partial_{t} V^{\varepsilon} + H\left(z,y,x,D_{z}V^{\varepsilon},\frac{1}{\varepsilon}D_{y}V^{\varepsilon}, \frac{1}{\varepsilon^{2}}D_{x}V^{\varepsilon}\right) = \ell(z),\quad \text{in } \mathds{R}^{m}\times\mathds{R}^{n}\times \mathds{R}^{\ell}\times (0,1),\\
    & V^{\varepsilon}(z,y,x,0) = \phi(z), \text{ in } \mathds{R}^{m}\times\mathds{R}^{n}\times \mathds{R}^{\ell}.
\end{aligned}
\right.
\end{equation}
The Hamiltonian $H$ is 
\begin{equation*}
\begin{aligned}
    & H(z,y,x,p,q,r)\\
    & = \sup\limits_{\alpha\in \Omega_{A},\beta\in \Omega_{B}, \gamma \in \Omega_{\Gamma}}\big\{-p\cdot (A_{0}(z)x+A_{1}(z)y + B_{1}(z)\alpha + C_{1}(z))\\
    & \quad \quad \quad \quad \quad \quad - q\cdot (A_{2}(z)y + B_{2}(z)\beta + C_{2}(z)) - r \cdot(A_{3}(z)x + B_{3}(z)\gamma + C_{3}(z))\big\}\\
    & = H_{1} + H_{2} + H_{3},
\end{aligned}
\end{equation*}
and 
\begin{equation*}
    \begin{aligned}
        H_{1} & = -(A_{0}(z)x + A_{1}(z)y + C_{1}(z))\cdot p + \sup\limits_{\alpha\in \Omega_A}\big\{ -p\cdot B_{1}(z)\alpha  \big \},\\
        H_{2} & = - (A_{2}(z)y  + C_{2}(z))\cdot q + \sup\limits_{\beta\in \Omega_B}\big\{ -q\cdot B_{2}(z)\beta \big \},\\
        H_{3} & = - (A_{3}(z)x + C_{3}(z))\cdot r + \sup\limits_{\gamma\in \Omega_\Gamma}\big\{ -r\cdot B_{3}(z)\gamma \big \}. 
    \end{aligned}
\end{equation*}
The term $H_{1}$ is the contribution of the macro-scale. The term $H_{2}$ is the contribution of the meso-scale. The term $H_{3}$ is the contribution of the micro-scale. 

We now consider the following asymptotic expansion.

\noindent \textbf{Ansatz. } $V^{\varepsilon}(t,z,y,x) = V_{0}(t,z) + \varepsilon V_{1}(y) + \varepsilon^{2} V_{2}(x) + O(\varepsilon^{3})$. 

This formally leads to three HJB equations, each one for a different scale: 
\begin{itemize}
    \item \textit{(The micro-scale)} we freeze $\bar{z},\; \bar{y}, \; \bar{p}:=\partial_{z}V_{0}(\bar{z}), \;  \bar{q}:=\partial_{y}V_{1}(\bar{y})$, and we solve the PDE whose unknown is the pair $(\lambda_{2}, V_{2}(\cdot))$ where $\lambda_{2}$ is a constant in $x$ (depending on $\bar{z}$ and on $\bar{p}$) and $V_{2}(\cdot)$ is a function of $x$ such that
    \begin{equation}\label{micro}
    \begin{aligned}
        & -(A_{3}(\bar{z})x+C_{3}(\bar{z}))\cdot \partial_{x}V_{2}(x) \\
        & \quad \quad \quad \quad + \sup\limits_{\gamma\in \Omega_\Gamma}\{-\partial_{x}V_{2}(x)\cdot B_{3}(\bar{z})\gamma\} - A_{0}(\bar{z})x\cdot \bar{p} = \lambda_{2},\quad \forall\, x.
    \end{aligned}
    \end{equation}
    \item \textit{(The meso-scale)} we freeze $\bar{z}, \; \bar{p}:=\partial_{z}V_{0}(\bar{z})$ and solve the PDE whose unknown is the pair $(\lambda_{1}, V_{1}(\cdot))$ where $\lambda_{1}$ is a constant in $y$ (depending on $\bar{z}$ and on $\bar{p}$) and $V_{1}(\cdot)$ is a function of $y$ such that
    \begin{equation}\label{meso}
    \begin{aligned}
        & - (A_{2}(\bar{z})y  + C_{2}(\bar{z}))\cdot \partial_{y}V_{1}(y) \\
        & \quad \quad \quad \quad + \sup\limits_{\beta\in \Omega_B}\big\{ -\partial_{y}V_{1}(y)\cdot B_{2}(\bar{z})\beta \big\} -A_{1}(\bar{z})y\cdot \bar{p}= \lambda_{1},\quad \forall\, y.
    \end{aligned}
    \end{equation}
    This is analogous to \eqref{cell}.
    \item \textit{(The macro-scale)} we solve the PDE whose unknown is the function $V_{0}$ and is the unique continuous viscosity solution in $\mathds{R}^{m}\times (0,1)$ of
    \begin{equation}\label{eff HJB multi}
    \left\{ \; 
    \begin{aligned}
        & \partial_{t}V_{0}(t,z) - C_{1}(z)\cdot \partial_{z}V_{0}(t,z) + \sup\limits_{\alpha\in \Omega_A}\big\{ -\partial_{z}V_{0}(t,z)\cdot B_{1}\alpha \big\}\\
        & \quad \quad \quad \quad \quad \quad \quad \quad \quad \quad  + \lambda_{1}(z,\partial_{z}V_{0}(t,z)) + \lambda_{2}(z,\partial_{z}V_{0}(t,z)) = \ell(z),\\
        & V_{0}(z,0) = \phi(z), \text{ in } \mathds{R}^{m}.
    \end{aligned}
    \right.
    \end{equation}
\end{itemize}

The PDE in \ref{eff HJB multi} is the one obtained when $\varepsilon\to 0$ in \eqref{HJB multi} (compare with \eqref{HJB 0}). To prove the convergence, we consider a cascaded approach as described in \cite[\S 4.1]{alvarez2008multiscale} (see also \cite{alvarez2007multiscale}). It goes as follows:

\begin{enumerate}[label=(\roman*)]
    \item Firstly, we consider the dynamics of $x$ being the only \textit{fast} one, while keeping frozen the dynamics of $z,y$. In order to pass to the limit, we solve \eqref{micro}, and obtain $(\lambda_{2}),V_{2})$ as we have previously done in Section \S \ref{sec: first order asymptot}. This is needed for the construction of the PDE problem in the last step. 
    \item Secondly, we freeze $z$ and  consider $y$ to be the \textit{fast} dynamics (note that at this stage, we have a system of $(z,y)$ only). We need now to solve \eqref{meso}, and obtain $(\lambda_{1},V_{1})$. This is the approach in Section \S \ref{sec: first order asymptot}. 
    \item Lastly, we consider the PDE \eqref{eff HJB multi} with $\lambda_{1}$ and $\lambda_{2}$ previously constructed, and  we can pass to the limit again as in \S \ref{sec: first order asymptot}.
\end{enumerate}
In other words, because the dynamics of the microscopic variable $x$ satisfies the same assumptions as of the one of the mesoscopic variable $y$, the results of the previous section apply so we can pass to the limit in the step (i) and recover a mesoscopic limit problem whose two-scaled variables are now $(z,y)$. The latter falls again in the framework of the two-scaled optimal control problem discussed previously. Thus we can repeat the same procedure to finally get the limit (macroscopic, or effective) problem whose solution is $V_{0}$ and the convergence of $V^{\varepsilon}$ to $V_{0}(t,z)$ is a consequence of \cite[Theorem 4.1]{alvarez2008multiscale}. 

\begin{thm}\label{thm: conv multi}
    With the standing assumptions (SA), the following statements hold.\vspace*{-3mm}
    \begin{enumerate}
        \item Let $z,y,p,q$ be fixed. There exists a unique $\lambda_{2}$ for which there exists a continuous periodic viscosity solution $V_{2}(\cdot)$ to \eqref{micro}.
        \item Let $z,p$ be fixed. There exists a unique $\lambda_{1}$ for which there exists a continuous periodic viscosity solution $V_{1}(\cdot)$ to \eqref{meso}. 
        \item As $\varepsilon\to 0$, the sequence $V^{\varepsilon}$ of solutions of \eqref{HJB multi} converges locally uniformly on $\mathds{R}^{m}\times (0, 1]$ to \eqref{eff HJB multi}. 
    \end{enumerate}
\end{thm}

Here, statement $(1)$ of the theorem concerns the point (i) above, statement $(2)$ is for (ii), and statement $(3)$ is for (iii). 

\begin{proof}
    The first and second statements are the same as Lemma \ref{lem: corrector}. The last statement is analog to Theorem \ref{thm conv} and is a consequence of \cite[Theorem 4.1]{alvarez2008multiscale}. 
    The convergence results (of the dynamics and the value function) hold for all the initial positions of the dynamics $y,x$, and uniformly on compact sets for $z$. This is ensured by the stability assumption.
\end{proof}

\begin{rem}
    We note that this cascaded approach is not only valid for a three-scale system, but can be performed for systems with multiple time-scales $\varepsilon, \varepsilon^{2}, \varepsilon^{3}, \dots \;$ . Indeed for a system with $N$ time-scale, it suffices to freeze all the variables until the $(N-1)$ time-scale, solve the last equation to obtain $(\lambda_{N-1}, V_{N-1}(\cdot))$. Then we solve the next one and obtain $(\lambda_{N-2},V_{N-2}(\cdot))$, and so on. Until the last PDE analogous to \eqref{eff HJB multi} where we would have the additional terms $\lambda_{1},\dots,\lambda_{N-1}$.
\end{rem}

\section{Application}\label{Sec: Application}

\subsection{Jin--Xin Two Scale Relaxation}

We start by recalling the relaxation due to Jin and Xin \cite{MR1322811,MR1793199,MR1693210}. The original PDE problem is
\begin{equation}\label{original}
\left\{\quad
    \begin{aligned}
        & \partial_{t}u + \partial_{x} (\mathcal{F}(u)) = 0,\\
        & u(x,0)=u_{0}(x).
    \end{aligned}
\right.
\end{equation}
A new variable $v$ is introduced and yields the following system of PDEs, for $a>0$,
\begin{equation}\label{relax JX}
\left\{\quad
\begin{aligned}
    & \partial_{t}u + \partial_{x} v = 0,\\
    & \partial_{t}v + a\,\partial_{x} u = -\frac{1}{\varepsilon}(v-\mathcal{F}(u)),\\
    & u(x,0)=u_{0}(x),\; v(x,0)=\mathcal{F}(u_{0}(x)).
\end{aligned}
\right.
\end{equation}
This is motivated by the small relaxation limit ($\varepsilon\to 0$) for which one recovers a local equilibrium $v=\mathcal{F}(u)$ and hence the original PDE problem. 

In order to study this convergence, we rewrite this relaxation by introduce another variable $\omega$ subject to the PDE
\begin{equation*}
\left\{\quad
    \begin{aligned}
        & \partial_{t}\omega = -\frac{1}{\varepsilon}(v-\mathcal{F}(u)),\\
        & \omega(x,0)=\omega_{0}(x),
    \end{aligned}
\right.
\end{equation*}
where $\omega_0$ will later be discussed. Then we define $\nu := v-\omega$ and write the system of PDEs satisfied by $(u,\nu,\omega)$ that is
\begin{equation}\label{eq: u nu oemga}
\left\{\quad
    \begin{aligned}
        & \partial_{t}u = - \partial_{x} \nu - \partial_{x}\omega, \quad && u(x,0)=u_{0}(x), \\
        & \partial_{t} \nu = - a\,\partial_{x} u , \quad && \nu(x,0)=\mathcal{F}(u_{0}(x)) -  \omega_{0}(x),\\
        & \partial_{t}\omega = -\frac{1}{\varepsilon}( \nu-\mathcal{F}(u) +\omega), \quad && \omega(x,0) = \omega_{0}(x).
    \end{aligned}
\right.
\end{equation}
The system satisfied by $(u,\nu,\omega)$ has the advantage of being \textit{singularly perturbed} in time only. This structure is easier to study, at least in the finite-dimensional setting which we discuss in the sequel. Note moreover that the system \eqref{eq: u nu oemga} still enjoys hyperbolicity. We have indeed
\begin{equation}\label{eq operator A}
    \partial_{t}
    \begin{pmatrix}
    u\\ \nu \\ \omega 
    \end{pmatrix}
    +
    \mathcal{A} \, 
    \partial_{x}
    \begin{pmatrix}
    u\\ \nu \\ \omega 
    \end{pmatrix}
    =  S(u,\nu,\omega),
\end{equation}
where 
\begin{equation}\label{eq: A and S}
    \mathcal{A} = 
    \begin{pmatrix}
        0 & 1 & 1 \\
        a & 0 & 0 \\
        0 & 0 & 0
    \end{pmatrix} 
    \quad 
    \text{ and }
    \quad 
    S(u,\nu,\omega) = 
    \begin{pmatrix}
       0 \\ 
       0 \\ 
       -\frac{1}{\varepsilon}( \nu-\mathcal{F}(u) +\omega) \\
    \end{pmatrix}.
\end{equation}
Note that $\mathcal{A}$ has the eigenvalues $\{0,\,  \sqrt{a}, -\sqrt{a}\}$ corresponding to the eigenvectors $(0, \, -1, \, 1)^{\top}$, $(a^{-1/2}, \, 1, \, 0)^{\top}$ and $(- a^{-1/2}, \, 1, \, 0)^{\top}$ respectively. It can then be expressed such that $\mathcal{A} = T\Lambda T^{-1}$ where
\begin{equation*}
    T = 
    \begin{pmatrix}
        0 & \sqrt{a} & -\sqrt{a}\\
        -1 & 1 & 1 \\
        1 & 0 & 0
    \end{pmatrix}, 
    \quad 
    \Lambda = 
    \begin{pmatrix}
        0 & 0 & 0 \\
        0 & \sqrt{a} & 0\\
        0 &  0 & -\sqrt{a}
    \end{pmatrix},
    \quad 
    T^{-1} = 
    \begin{pmatrix}
        0 & 0 & 1\\
        \sqrt{a}/2 & 1/2 & 1/2\\
        -\sqrt{a}/2 & 1/2 & 1/2
    \end{pmatrix}.
\end{equation*}
Therefore we can write \eqref{eq operator A} in the form
\begin{equation}\label{eq operator A 2}
    \partial_{t} T^{-1}
    \begin{pmatrix}
    u\\ \nu \\ \omega 
    \end{pmatrix}
    +
    \Lambda \, 
    \partial_{x} T^{-1}
    \begin{pmatrix}
    u\\ \nu \\ \omega 
    \end{pmatrix}
    =  T^{-1}S(u,\nu,\omega).
\end{equation}
Finally, setting 
\begin{equation*}
    \xi = 
    \begin{pmatrix}
    \xi_{1}\\ \xi_{2} \\ \xi_{3}    
    \end{pmatrix} 
    :=  T^{-1} 
    \begin{pmatrix}
        u\\ \nu \\ \omega
    \end{pmatrix}
    =
    \begin{pmatrix}
        \omega\\
        \frac{\sqrt{a}}{2} u + \frac{1}{2}(\nu + \omega)\\
        -\frac{\sqrt{a}}{2} + \frac{1}{2}(\nu + \omega)
    \end{pmatrix}
    \quad \text{ and } \quad
    \bar{S}(\xi) := T^{-1} S(u,\nu,\omega),
\end{equation*}
yields the following equivalent system to \eqref{eq operator A}
\begin{equation}
    \left\{\quad 
    \begin{aligned}
        & \partial_{t} \xi_{1}  =     \bar{S}_{1}(\xi),\\
        & \partial_{t} \xi_{2} + \sqrt{a}\, \partial_{x}\,\xi_{2} = \bar{S}_{2}(\xi),\\
        & \partial_{t} \xi_{3} - \sqrt{a}\, \partial_{x}\,\xi_{3} = \bar{S}_{3}(\xi).
    \end{aligned}
    \right.
\end{equation}


We now consider a semi-discretisation in space of the PDE system \eqref{eq: u nu oemga} satisfied by $(u,\nu,\omega)$. We introduce $\mathbf{u}_{\varepsilon}(\cdot), \mathbf{v}_{\varepsilon}(\cdot), \mathbf{w}_{\varepsilon}(\cdot)\in \mathds{R}^{m}$ satisfying the system of ODEs, for $s\in [0,T]$ and $T<+\infty$
\begin{equation*}
\left\{\quad
    \begin{aligned}
        & \dot{\mathbf{u}}_{\varepsilon}(s) = -D\mathbf{v}_{\varepsilon}(s) -D\mathbf{w}_{\varepsilon}(s), \quad && \mathbf{u}_{\varepsilon}(0) = \mathbf{u}_0\in \mathds{R}^{m}, \\
        & \dot{\mathbf{v}}_{\varepsilon}(s) = -a\, D \mathbf{u}_{\varepsilon}(s), \quad && \mathbf{v}_{\varepsilon}(0)=\mathbf{v}_0\in \mathds{R}^{m}, \\
        & \dot{\mathbf{w}}_{\varepsilon}(s) = -\frac{1}{\varepsilon}\big[ \mathbf{v}_{\varepsilon}(s)-\mathcal{F}(\mathbf{u}_{\varepsilon}(s)) +\mathbf{w}_{\varepsilon}(s)\big], \quad && \mathbf{w}_{\varepsilon}(0) = \mathbf{w}_0\in \mathds{R}^{m}.
    \end{aligned}
\right. 
\end{equation*}
We denote by $D$,  e.g. a first-order finite-volume spatial discretisation of the transport operator $\partial_x$. Its specific structure is for the following considerations not relevant.  The $i$-th component of the vectors $\mathbf{u}_{\varepsilon}(\cdot), \mathbf{v}_{\varepsilon}(\cdot), \mathbf{w}_{\varepsilon}(\cdot)\in \mathds{R}^{m}$ are the values of the solution $(u,v,w)$ at the cell centres $x_i = i \Delta x$ for some spatial grid $\Delta x>0.$ For details on the discretisation we refer to \cite{MR1322811}. The previous system falls within the general theory of \textit{singular perturbation}, or singularly perturbed system of ODEs for which one distinguishes between the \textit{slow} dynamics (here $\mathbf{u}_{\varepsilon}$ and $\mathbf{v}_{\varepsilon}$) and the \textit{fast} dynamics (here $\mathbf{w}_{\varepsilon}$). 

Furthermore, we introduce control parameters $\alpha(s)\in \Omega_A$ and $\beta(s)\in \Omega_B$,  with $\Omega_{A}\subset \mathds{R}^{p}, \Omega_{B} \subset  \mathds{R}^{q}$ are compact and convex and $p,q>0$, as follows
\begin{equation}\label{ODE eps app}
\left\{\;
    \begin{aligned}
        & \dot{\mathbf{u}}_{\varepsilon}(s) = -D\mathbf{v}_{\varepsilon}(s) -D\mathbf{w}_{\varepsilon}(s) + \mathds{H}(\mathbf{u}_{\varepsilon}(s))\alpha(s),  && \mathbf{u}_{\varepsilon}(0) = \mathbf{u}_0, \\
        & \dot{\mathbf{v}}_{\varepsilon}(s) = -a\, D \mathbf{u}_{\varepsilon}(s) , && \mathbf{v}_{\varepsilon}(0)=\mathbf{v}_0, \\
        & \dot{\mathbf{w}}_{\varepsilon}(s) = -\frac{1}{\varepsilon}\big[ \mathbf{v}_{\varepsilon}(s)-\mathcal{F}(\mathbf{u}_{\varepsilon}(s)) - \mathds{G}(\mathbf{u}_{\varepsilon}(s))\beta(s) +\mathbf{w}_{\varepsilon}(s)\big],  && \mathbf{w}_{\varepsilon}(0) = \mathbf{w}_0,
    \end{aligned}
\right.
\end{equation}
where $\mathbf{u}_0,\mathbf{v}_0,\mathbf{w}_0\in \mathds{R}^{m}$, $\mathds{H}$ and $\mathds{G}$ are matrices functions which multiply the control parameters $\alpha,\beta$. 

\begin{rem}
    In the notation of \eqref{eq operator A}, this corresponds to adding $\mathds{H}(u)\alpha$ and $\frac{1}{\varepsilon}\mathds{G}(u)\beta$ to the first and third entries of $S(u,\nu,\omega)$ in \eqref{eq: A and S} respectively.
\end{rem}

These control parameters will be subject to minimising a given cost function 
\begin{equation}\label{eq: value eps JX}
\tag{SP.1}
    \inf \; \Phi(\mathbf{u}_{\varepsilon}(T)),\quad \text{s.t. } \; \eqref{ODE eps app}\,,
\end{equation}
where we assume the cost $\Phi$  to be continuous. The limit $\varepsilon\to 0$  yields the reduced control problem
\begin{equation}\label{eq: value eff JX}
\tag{R.1}
    \inf \; \Phi(\mathbf{u}(T)),\quad \text{s.t. } \; \eqref{ODE eff app}\,,
\end{equation}
where the reduced dynamics is
\begin{equation}\label{ODE eff app}
\left\{\;
\begin{aligned}
     & \dot{\mathbf{u}}(s) + D\big[\mathcal{F}(\mathbf{u}(s))+\mathds{G}(\mathbf{u}(s))\beta(s)\big] = \mathds{H}(\mathbf{u}(s))\alpha(s),\quad s\in (0,T],\\
     & \mathbf{u}(0) = u_{0}\in\mathds{R}^{m},
\end{aligned}
\right.
\end{equation}
together with
\begin{equation}\label{eq: v eff 1}
    \dot{\mathbf{v}}_{\varepsilon}(s) = -a\, D \mathbf{u}_{\varepsilon}(s) , \quad \quad \mathbf{v}_{\varepsilon}(0)=\mathbf{v}_0.
\end{equation}
Since in the latter control problem only the function $\mathbf{u}$ is involved,  we can omit the dynamics for $\mathbf{v}$. 

The ODE \eqref{ODE eff app} is obtained by first solving for $\mathbf{w}_{\varepsilon}$ the algebraic equation
\begin{equation*}
    0= \mathbf{v}_{\varepsilon}(s)-\mathcal{F}(\mathbf{u}_{\varepsilon}(s)) - \mathds{G}(\mathbf{u}_{\varepsilon}(s))\beta(s) +\mathbf{w}_{\varepsilon}(s),
\end{equation*}
then substituting the value of $\mathbf{w}_{\varepsilon}$ in the dynamics for $\mathbf{u}_{\varepsilon}$. 

Note that \eqref{ODE eff app} is the discretisation of 
\begin{equation}\label{eq: original pde}
\left\{\quad
    \begin{aligned}
        & \partial_{t}u(x,t) + \partial_{x}[\mathcal{F}(u(x,t)) + \mathds{G}(u(x,t))\beta(x,t)] = \mathds{H}(u(x,t))\alpha(x,t),\\
        & u(x,0)=u_{0}(x),
    \end{aligned}
\right.
\end{equation}
which is the controlled version of the original PDE problem \eqref{original} with the additional control terms arising from $\mathds{G}$ and $\mathds{H}$. 

The limit as $\varepsilon\to 0$ is summarised in the following result whose proof is a direct application of Theorem \ref{thm: conv G}.

\begin{cor}\label{cor: JX}
    As $\varepsilon\to 0$, the problem \eqref{eq: value eps JX} is approximated by \eqref{eq: value eff JX} in the sense of Definition \ref{def:approx}. In particular, the controlled system \eqref{ODE eps app} with any given initial conditions converges to \eqref{ODE eff app} locally uniformly on any finite time interval, and we have $\inf\, \Phi(\mathbf{u}_{\varepsilon}(T)) \to \inf \, \Phi(\mathbf{u}(T))$. 
\end{cor} 

Moreover, the convergence in the above corollary extends to the HJB equations whose solutions are the value functions of \eqref{eq: value eps JX} and \eqref{eq: value eff JX} respectively, in the sense of Theorem \ref{thm conv}.

\subsection{Jin--Xin Three--Scale Relaxation}\label{sec: JX 3 app}

In this section, we would like to go one step further in the analysis by considering a multi-scale approximation of \eqref{original} and \eqref{relax JX}. Assuming a decomposition
\begin{equation}
\label{eq: F decomp}
    \mathcal{F}(u) \approx F_{0}(u) + \varepsilon\mathcal{F}_{1}(u),
\end{equation} 
 the  following relaxation system for $a,b>0$ is obtained 
\begin{equation}\label{relax JX 2}
\left\{\quad
\begin{aligned}
    & \partial_{t}u + \partial_{x} v = 0,\\
    & \partial_{t}v + a\,\partial_{x} u = -\frac{1}{\varepsilon}\big[ v-(\mathcal{F}_{0}(u) + \varepsilon \omega)\big],\\
    & \partial_{t}\omega + b\,\partial_{x} u = -\frac{1}{\varepsilon^{2}}\big[\omega-\mathcal{F}_{1}(u)\big],\\
    & u(x,0)=u_{0}(x),\; v(x,0)=\mathcal{F}_{0}(u_{0}(x)), \; \omega(x,0) = \mathcal{F}_{1}(u_{0}(x)).
\end{aligned}
\right.
\end{equation}
This system exhibits three scales: the first PDE is in the \textit{macro-scale}, the second one is in  the \textit{meso-scale}, the third one is in the \textit{micro-scale}. In the latter PDE, the local equilibrium is reached for $\omega = \mathcal{F}_{1}(u)$ and yields the term $F_{0}(u) + \varepsilon\mathcal{F}_{1}(u)$ in the second PDE. 

Our goal is to capture the effect of $\mathcal{F}_{0}$ and $\mathcal{F}_{1}$ from \eqref{eq: F decomp} on the value function of the control problem for the finite-dimensional (discrete) version of \eqref{relax JX 2}.

Following the same idea as in the previous subsection, we introduce new variables $p(x,t),q(x,t)$ whose time derivatives correspond to the right-hand side of the second and third equations. Starting from
\begin{equation*}
    \partial_{t}v + a\,\partial_{x} u = -\frac{1}{\varepsilon}\big[ v-(\mathcal{F}_{0}(u) + \varepsilon \omega)\big] = -\frac{1}{\varepsilon}\big[ v -\mathcal{F}_{0}(u) \big] + \omega,
\end{equation*}
we set $\; \partial_{t} p = -\frac{1}{\varepsilon}\big[ v -\mathcal{F}_{0}(u) \big],$
then obtain $\; \partial_{t}v + a\,\partial_{x} u = \partial_{t} p + \omega,$
which is
\begin{equation*}
    \partial_{t}[v-p] = - a\,\partial_{x} u + \omega.
\end{equation*}
Similarly, we start by setting $\, \partial_{t} q = -\frac{1}{\varepsilon^{2}}\big[\omega-\mathcal{F}_{1}(u)\big],$
then get $\, \partial_{t}[\omega-q] =- b\,\partial_{x} u.$
Ultimately, using $\nu := v-p$ and $\text{w} := \omega-q$ yields the following three--scale system
\begin{equation}\label{eq: multi scale}
\left\{\quad
    \begin{aligned}
        & \partial_{t}u = - \partial_{x} \nu - \partial_{x}p, \quad && u(x,0)=u_{0}(x), \\
        & \partial_{t} \nu = - a\,\partial_{x} u + \text{w} + q, \quad && \nu(x,0)= \mathcal{F}_{0}(u_{0}(x))  - p_{0}(x),\\
        & \partial_{t}\text{w} = -b\, \partial_{x} u, \quad && \text{w}(x,0) = \mathcal{F}_{1}(u_{0}(x)) - q_{0}(x),\\
        & \partial_{t}p = -\frac{1}{\varepsilon}(\nu - \mathcal{F}_{0}(u) + p), \quad && p(x,0) = p_{0}(x),\\
        & \partial_{t}q = -\frac{1}{\varepsilon^{2}}(\text{w} - \mathcal{F}_{1}(u) + q), \quad  && q(x,0) = q_{0}(x).
    \end{aligned}
\right.
\end{equation}
This system enjoys hyperbolicity as it can be expressed such that 
\begin{equation}\label{eq operator A tilde}
    \partial_{t}
    \begin{pmatrix}
        u \\\nu \\ \text{w} \\ p \\ q
    \end{pmatrix}
    +
    \widetilde{\mathcal{A}}\;
    \partial_{x}
    \begin{pmatrix}
        u \\ \nu \\ \text{w} \\ p \\ q
    \end{pmatrix} 
    = \widetilde{S}(u,\nu,\text{w}, p, q),
\end{equation}
where 
\begin{equation*}
    \widetilde{\mathcal{A}} = 
    \begin{pmatrix}
        0 & 1 & 0 & 1 & 0 \\
        a & 0 & 0 & 0 & 0\\
        b & 0 & 0 & 0 & 0 \\
        0 & 0 & 0 & 0 & 0
    \end{pmatrix}
    \quad \text{ and } \quad
    \widetilde{S}(u,\nu,\text{w},p , q) = 
    \begin{pmatrix}
        0 \\
        \text{w} + q \\
        0\\
        -\frac{1}{\varepsilon}(\nu - \mathcal{F}_{0}(u) + p) \\
        -\frac{1}{\varepsilon^{2}}(\omega - \mathcal{F}_{1}(u) + q)
    \end{pmatrix}. 
\end{equation*}
Observe that $\widetilde{A}$ has eigenvalues $\{0, 0, 0, \sqrt{a}, -\sqrt{a}\}$ and  eigenvectors $(0,0,1,0,0)^{\top}$, $(0,-1,0, 1,0)^{\top}$, $(0,0,0,0,1)^{\top}$, $(\sqrt{a}/b, a/b,1,0,0)^{\top}$ and $(-\sqrt{a}/b, a/b,1,0,0)^{\top}$ respectively. It can then be expressed as $\widetilde{A} = \widetilde{T} \widetilde{\Lambda} \widetilde{T}^{-1}$ where we skip the details of $\widetilde{T}$. The matrix  $\widetilde{\Lambda}$ is diagonal with three diagonal entries being zero and $\pm \sqrt{a}$. 

Then, we write \eqref{eq operator A tilde} in the form (compare with \eqref{eq operator A 2})
\begin{equation}
    \label{eq operator A tilde 2}
    \partial_{t} \widetilde{T}^{-1} 
    \begin{pmatrix}
        u\\ \nu \\ \text{w} \\ p \\ q
    \end{pmatrix}
    + \widetilde{\Lambda} \, \partial_{x} \widetilde{T}^{-1} 
    \begin{pmatrix}
        u \\ \nu \\ \text{w} \\ p \\ q
    \end{pmatrix}
    = 
    \widetilde{T}^{-1}\widetilde{S}(u,\nu,\text{w},p,q).
\end{equation}
Setting
\begin{equation*}
    \widetilde{\xi} = 
    \begin{pmatrix}
        \tilde{\xi}_{1} \\ \tilde{\xi}_{2} \\ \tilde{\xi}_{3} \\ \tilde{\xi}_{4} \\ \tilde{\xi}_{5}
    \end{pmatrix}
    := \widetilde{T}^{-1}
    \begin{pmatrix}
        u \\ \nu \\ \text{w} \\ p \\ q 
    \end{pmatrix}
    =
    \begin{pmatrix}
        \frac{-b}{a}(\nu +p) + \text{w} \\
        p \\
        q \\
        \frac{b}{2 \sqrt{a}}u + \frac{b}{2a} (\nu + p)\\
        \frac{-b}{2 \sqrt{a}}u + \frac{b}{2 a}(\nu + p)
    \end{pmatrix},
\end{equation*}
yields the following equivalent system to \eqref{eq operator A tilde}
\begin{equation}
    \partial_{t} \widetilde{\xi} + \widetilde{\Lambda} \, \partial_{x}\widetilde{\xi} = \widetilde{T}^{-1} \widetilde{S}(u,\nu,\text{w},p,q).
\end{equation}

The finite-dimensional version of \eqref{eq: multi scale} is the following system of ODEs to which we have added three control parameters, $\alpha(s)\in \Omega_A$, $\beta(s)\in \Omega_B$ and $\gamma(s)\in \Omega_\Gamma$, where $\Omega_{A},\Omega_{B}, \Omega_{\Gamma}$ are three compact and convex sets
\begin{equation}\label{relax 2 eps}
\left\{
    \begin{aligned}
        & \dot{\mathbf{u}}_{\varepsilon}(s) = -D\mathbf{v}_{\varepsilon}(s) -D\mathbf{p}_{\varepsilon}(s) + \mathds{H}(\mathbf{u}_{\varepsilon}(s))\alpha(s), \; && \mathbf{u}_{\varepsilon}(0) = \mathbf{u}_0\in \mathds{R}^{m}, \\
        & \dot{\mathbf{v}}_{\varepsilon}(s) = -a\, D \mathbf{u}_{\varepsilon}(s) + \mathbf{w}_{\varepsilon} + \mathbf{q}_{\varepsilon}, \; && \mathbf{v}_{\varepsilon}(0)=\mathbf{v}_0\in \mathds{R}^{m}, \\
        & \dot{\mathbf{w}}_{\varepsilon}(s) = - b \, D \mathbf{u}_{\varepsilon}(s), \; && \mathbf{w}_{\varepsilon}(0)=\mathbf{w}_0\in \mathds{R}^{m}, \\
        & \dot{\mathbf{p}}_{\varepsilon}(s) = -\frac{1}{\varepsilon}\big[ \mathbf{v}_{\varepsilon}(s)-\mathcal{F}_{0}(\mathbf{u}_{\varepsilon}(s)) -  \mathds{G}(\mathbf{u}_{\varepsilon}(s))\beta(s) +\mathbf{p}_{\varepsilon}(s)\big], \; && \mathbf{p}_{\varepsilon}(0) = \mathbf{p}_0\in \mathds{R}^{m},\\
        & \dot{\mathbf{q}}_{\varepsilon}(s) = -\frac{1}{\varepsilon^{2}}\big[ \mathbf{w}_{\varepsilon}(s)-\mathcal{F}_{1}(\mathbf{u}_{\varepsilon}(s))  - \mathds{K}(\mathbf{u}_{\varepsilon}(s))\gamma(s)+\mathbf{q}_{\varepsilon}(s)\big], \; && \mathbf{q}_{\varepsilon}(0) = \mathbf{q}_0\in \mathds{R}^{m}.
    \end{aligned}
\right.
\end{equation}

The limiting (effective) system as $\varepsilon\to 0$, is again given by \eqref{ODE eff app}. The system \eqref{relax 2 eps} falls within the framework of section \ref{sec: multiscale sys} since it can be written in the form \eqref{CSP 2} as follows
with the dimensions $m=3d$, $n=d$, $\ell=d$, where $d$ is the space dimension (discretisation).  The variables are 
$z=[\textbf{u}_{\varepsilon}^{\top},\textbf{v}_{\varepsilon}^{\top},\textbf{w}_{\varepsilon}^{\top}]^{\top}$, \quad $y=\textbf{p}_{\varepsilon}$, \quad $x=\textbf{q}_{\varepsilon}$.  Let $0_{d}, 1_{d}$ be the $d$-dimensional column vectors whose entries are all $0$ and $1$ respectively. Let also $a,b>0$ be constant parameters and $D$ the discretisation matrix of transport operator as before. Then, we set
\begin{equation*}
    \begin{aligned}
        & A_{0} =
        \begin{pmatrix}
            0_{d} \\ 1_{d} \\ 0_{d}
        \end{pmatrix}, \; 
        A_{1} = 
        \begin{pmatrix}
            -D \\ 0_{d} \\ 0_{d}
        \end{pmatrix}, \;
        B_{1} = 
        \begin{pmatrix}
            \mathds{H}(\mathbf{u}_{\varepsilon})\\ 0_{d} \\ 0_{d}
        \end{pmatrix},\;
        C_{1} = 
        \begin{pmatrix}
            0_{d} & -D & 0_{d} \\
            -a\,D & 0_{d} & 1_{d} \\
            -b\,D & 0_{d} & 0_{d}
        \end{pmatrix}\!\!
        \begin{pmatrix}
        \mathbf{u}_{\varepsilon} \\ \mathbf{v}_{\varepsilon} \\ \mathbf{w}_{\varepsilon}
    \end{pmatrix},\\
    & A_{2} = -1_{d}, \quad
    B_{2} = \mathds{G}(\mathbf{u}_{\varepsilon}),\quad C_{2} = \mathcal{F}_{0}(\mathbf{u}_{\varepsilon}) + 
    \begin{pmatrix}
        0_{d}^{\top} & -1_{d}^{\top} & 0_{d}^{\top}
    \end{pmatrix}\!\!
    \begin{pmatrix}
        \mathbf{u}_{\varepsilon} \\ \mathbf{v}_{\varepsilon} \\ \mathbf{w}_{\varepsilon}
    \end{pmatrix},\\
    & A_{3} = -1_{d}, \quad B_{3} = \mathds{K}(\mathbf{u}_{\varepsilon}), \quad C_{3} = \mathcal{F}_{1}(\mathbf{u}_{\varepsilon}) + 
    \begin{pmatrix}
        0_{d}^{\top} & 0_{d}^{\top} & -1_{d}^{\top} 
    \end{pmatrix}\!\!
    \begin{pmatrix}
        \mathbf{u}_{\varepsilon} \\ \mathbf{v}_{\varepsilon} \\ \mathbf{w}_{\varepsilon}
    \end{pmatrix}.
    \end{aligned}
\end{equation*}

Note that the multiscale approach yields at the limit $\varepsilon\to 0$ the same effective ODE system \eqref{ODE eff app}, the difference being in the dynamics of $\mathbf{v}$ (compare with \eqref{eq: v eff 1}) 
\begin{equation}\label{eq: v eff 2}
    \dot{\mathbf{v}}(s) = - a\, D\mathbf{u}(s) + \mathcal{F}_{1}(\mathbf{u}(s)) + \mathds{K}(\mathbf{u}(s))\gamma(s), \quad \quad \mathbf{v}(0) = \mathbf{v}_{0}.
\end{equation}
and the new variable $\mathbf{w}$ governed by
\begin{equation}\label{eq: w eff 2}
    \dot{\mathbf{w}}(s) = - b \, D \mathbf{u}(s),\quad \quad \mathbf{w}(0) = \mathbf{w}_{0}.
\end{equation}
However, the benefit of using three (or more) scales appears in the asymptotic expansion of the value function which keeps track of the higher-order terms in \eqref{eq: F decomp}. Thus, we make a corollary of Theorem \ref{thm: conv multi} on the convergence at the level of the HJB equations. 

Let $\ell, \phi$ be as in the standing assumptions \textbf{(SA)} of \S \ref{sec: hjb}. Consider the optimal control problem
\begin{equation}\label{SP JX multi}
\tag{SP.2}
    \mathbf{V}^{\varepsilon}(\mathbf{u}_{0}, \mathbf{v}_{0}, \mathbf{w}_{0}, \mathbf{p}_{0}, \mathbf{q}_{0}) = 
    \inf \; \int_{0}^{1} \ell(\mathbf{u}_{\varepsilon}(s))\,\text{d}s +  \phi(z(1)), \; \text{ s.t.: } \;  \eqref{relax 2 eps} \text{ with } s\in [0,1].
\end{equation}
And consider the corresponding reduced problem
\begin{equation}\label{SP JX multi R}
\tag{R.2}
    \mathbf{V}_{0}(\mathbf{u}_{0}) =
    \inf \; \int_{0}^{1} \ell(\mathbf{u}(s))\,\text{d}s +  \phi(z(1)), \; \text{ s.t.: } \;  \eqref{ODE eff app} \text{ with } s\in [0,1].
\end{equation}

\begin{rem}\label{rem: cost general}
    Observe that $\ell,\phi$ being dependent on $\mathbf{u}$ is a choice we made for simplicity only. Indeed the previous theoretical results apply for control problems where the cost functional depends on the slow variables (therein denoted by $z(\cdot)$) and which are in \eqref{relax 2 eps} the variables  $\mathbf{u}_{\varepsilon}, \mathbf{v}_{\varepsilon},\mathbf{w}_{\varepsilon}$. Hence, we could also consider $\ell,\phi$ as functions of $\mathbf{u},\mathbf{v},\mathbf{w}$ in which case, the dynamics \eqref{ODE eff app} would be complemented with \eqref{eq: v eff 2} and \eqref{eq: w eff 2} and the value function would be $\mathbf{V}_{0}(\mathbf{u}_{0}, \mathbf{v}_{0}, \mathbf{w}_{0})$.
\end{rem}

A consequence of Theorem \ref{thm: conv multi} is the following.

\begin{cor}\label{cor JX 3}
    The value function $\mathbf{V}^{\varepsilon}$ of \eqref{SP JX multi} converges locally uniformly on $\big(\mathds{R}^{m}\big)^{5}\times (0,1]$ to the value function $\mathbf{V}_{0}$ of \eqref{SP JX multi R}.
\end{cor}

In other words, the above corollary states that it would be sufficient to rely on the control produced for the singularly perturbed dynamics, to approximate the control of \eqref{ODE eff app}. This control is a discretisation of \eqref{eq: original pde}. Moreover $\mathbf{V}^{\varepsilon}$ admits the asymptotic expansion described in \S \ref{sec: linear 2} where the dominant term is $\mathbf{V}_{0}$.

\section{Summary}

The object of the present manuscript is twofold.  First, we provide a new formulation for controlled stiff differential equations based on singular perturbations. The convergence of the value functions with respect  to the stiffness parameter is shown. An asymptotic expansion of the value function corresponding to the control problem is formulated. This is useful in particular when designing  feedback controls  depending on the gradient of the value function.  The second contribution of the manuscript is the higher order approximation of controlled stiff differential equations. Consequently, this leads to a  higher order corrections of the value function. We show how our results apply to Jin--Xin relaxation system with additionally a control parameter. We also expand Jin--Xin relaxation system to allow for higher--order formulation and approximation. 

Finally, it is worth mentioning that an application of the present framework would be highly desirable in the context of multi-agent systems and for their control. In the example of systems with human-in-the-loop, or unmanned aerial vehicle (UAV), this might refer to the automatic control of the system compared with the human action, which in case of gain disturbance would require a sensor that is sufficiently fast compared to the dynamics of the systems.

\subsection*{Acknowledgement} The authors are grateful to the referees for their careful reading, and for their comments which helped improve the manuscript.

\appendix
\section{Limit Occupational Measure Sets}\label{appendix}

Our assumptions allow us to obtain a limiting problem (i.e. the leading term in the expansion when $\varepsilon\to 0$) that can be easily constructed by setting $\varepsilon =0$ in \eqref{eq: CSP intro 1} and solving for $\mathbf{y}$ the equation $G(\mathbf{z}(s),\mathbf{y}(s),\beta(s),s)=0$. This construction is desirable in view of our motivating applications. But for general optimal control problems, such limit is wrong, and a more sophisticated limiting problem is usually obtained as follows. 

Denote by $\lambda\{I\}$ the Lebesgue measure of the interval $I$, and let $y_{z}$ follow the same dynamics as $y$ in \eqref{CSP} but where $\varepsilon=1$ and $z$ is frozen (constant). 
We define the \textit{occupational measures} (see \cites{gaitsgory1992suboptimization, homocp,gaitsgory1999limit, kouhkouh22phd}) as
\begin{equation*}
    \varphi_{z}^{(y_{0},\beta, T)}(Q) := \frac{1}{T}\lambda\{s \in [0,T] \, | \, (y_{z}(s),\beta(s)) \in Q\} ,
\end{equation*}
for any $Q$ a Borel subset of $\mathds{T}^{n}\times \Omega_{B}$. 
We denote the union of such measures over all admissible controls
\begin{equation*}
    \Psi(z,T,y_{0}) := \bigcup_{\beta}\big\{\varphi_{z}^{(y_{0},\beta, T)}\big\}.
\end{equation*}
Then for all $z\in Z,\, y_{0}\in \mathds{T}^{n}$ where $Z$ is a compact set (see remark \ref{rem periodic}), the limit in the Hausdorff metric as $T\to +\infty$ is
\begin{equation}\label{LOMS}
    \lim\limits_{T\to +\infty}\Psi(z,T,y_{0}) = \mathfrak{L}(z),\quad \forall\, z\in Z,\, y_{0}\in \mathds{T}^{n},
\end{equation}
called the \textit{limit occupational measure set} (LOMS) whose existence is discussed in \cite{gaitsgory1999limit} and references therein. Also define the set
\begin{equation}
    V_{f}(z) := \left\{ v\, \bigg|\, v = \int_{\mathds{T}^{n}\times \Omega_{B}} f(z,y,\alpha, \beta)\,\varphi(\dd y \times \dd \beta),\; \varphi \in \mathfrak{L}(z),\; \alpha\in \Omega_{A} \right\},
\end{equation}
where $f$ here denotes the dynamics of the slow variable $z$, which in the situation of \eqref{CSP} is given by $f(z,y,\alpha, \beta) = A_{1}(z)y + B_{1}(z)\alpha + C_{1}(z)$.

A result from \cite{gaitsgory1999limit} (see also \cite{gaitsgory1992suboptimization}) ensures the existence of a function $\mu(\varepsilon)$ such that $\,\lim\limits_{\varepsilon\to 0} \, \mu(\varepsilon)=0,\,$ 
and corresponding to any trajectory $(z_{\varepsilon}(\cdot),y_{\varepsilon}(\cdot))$ of \eqref{CSP}, there exists a solution $\tilde{z}(\cdot)$ of 
\begin{equation}\label{limit sys}
    \dot{\tilde{z}}(s) \in V_{f}(\tilde{z}(s)),\quad \tilde{z}(0)=z_{0},
\end{equation}
such that
\begin{equation}
    \max\limits_{s\in[0,T]}\|z_{\varepsilon}(s) - \tilde{z}(s)\| \leq \mu(\varepsilon).
\end{equation}
Conversely, corresponding to any solution of \eqref{limit sys} there exists a trajectory of \eqref{CSP} such that the same inequality is satisfied.

Concerning the HJB equation, when $\varepsilon\to 0$, the \textit{effective} (limit) PDE problem obtained in \cite[Theorem 3.1]{homocp} and \cite[
Theorem 1.3.1]{kouhkouh22phd} is 
\begin{equation}
    \left\{ \; 
    \begin{aligned}
        & \partial_{t}\widetilde{v} + \widetilde{H}(z,D_{z}\widetilde{v}) = \ell(z),\; \text{ in } \mathds{R}^{m}\times(0,T),\\
        & \widetilde{v}(z,0) = \phi(z),\; \text{ in } \mathds{R}^{m},
    \end{aligned}
    \right.
\end{equation}
for $T\in (0,+\infty)$, where the \textit{effective} Hamiltonian (compare with $\widetilde{H}$ in the proof of Theorem \ref{thm conv}) is
\begin{equation}\label{eff Ham}
    \widetilde{H}(z,p) = \sup\limits_{\alpha\in \Omega_{A}, \mu \in \mathfrak{L}(z)}\{ - p\cdot \widetilde{f}(z,\alpha,\mu) \},
\end{equation}
and the \textit{effective} dynamics is
\begin{equation}
    \widetilde{f}(z, \alpha,\mu) := \int_{\mathds{T}^{n}\times \Omega_{B}} f(z,y,\alpha,\beta)\,\mu(\dd y\times \dd \beta).
\end{equation}

\begin{rem}
    A particular example of a measure $\mu_{\circ}$ belonging to the LOMS \eqref{LOMS} is the product of two Dirac measures $\mu_{\circ}:=\delta_{\bar{y}}(\text{d}y)\otimes\delta_{\bar{\beta}}(\text{d}\beta)$. In this case, the resulting dynamics for the slow process becomes
    \begin{equation*}
         \widetilde{f}(z, \alpha,\mu_{\circ}) := \int_{\mathds{T}^{n}\times \Omega_{B}} f(z,y,\alpha,\beta)\,\mu_{\circ}(\dd y\times \dd \beta) = f(z,\bar{y},\alpha,\bar{\beta}).
    \end{equation*}
    This corresponds to the \textit{reduced} case which would be our situation. \\
    \indent In this sense, the reduced dynamics (obtained when $\mu$ is a product of Diracs) is a particular example of the averaged dynamics (obtained for any measure $\mu\in \mathfrak{L}(z)$ the LOMS in \eqref{LOMS}). In general, the LOMS contains much more measures than just the product of two Diracs. 
\end{rem}

Moreover, $\widetilde{v}$ is the value function of the (\textit{effective}) optimal control problem 
\begin{equation}\label{averaged ocp}
\begin{aligned}
    \widetilde{v}(z,t) =&  \inf\limits_{\alpha,\mu}\, \int_{0}^{t}{\ell}(\tilde{z}(s)\,\dd s + \phi(\tilde{z}(t))\\
    & \; \text{ subject to: } \dot{\tilde{z}}(s) = \widetilde{f}(\tilde{z}(s), \alpha(s),\mu(s)),\quad \tilde{z}(0) = z_{0},\\
    &\quad \quad \,\; \text{ and } \quad \mu(s) \in \mathfrak{L}(\tilde{z}(s)), \quad \alpha(s)\in \Omega_{A}.
\end{aligned}
\end{equation}
The dynamics can be written as a differential inclusion 
\begin{equation*}
    \dot{\Tilde{z}}(s) \in \widetilde{f}(\tilde{z}(s), \alpha(s), \mathfrak{L}(\tilde{z}(s))), 
\end{equation*}
which is also \eqref{limit sys}. 
This is reminiscent to relaxed optimal controls. 

In our setting, the LOMS reduces to a singleton as stated in the proof of Theorem \ref{thm conv}. Hence, the value function $V_{0}$ in Theorem \ref{thm conv} coincides with $\widetilde{v}$ above, and the limiting dynamics $\bar{z}(\cdot)$ in \eqref{reduced} coincides with $\tilde{z}(\cdot)$. \\
It is clear that the control problem \eqref{value eff} is much simpler than \eqref{averaged ocp}.

\section{Examples of application}\label{appendix examples}

\subsection{Goldstein-Taylor two--scale model}\label{sec: GT 2}

We consider the Goldstein--Taylor model in the formulation of \cite{albi2014asymptotic}. The latter describes the time evolution of two-particle densities $f^{+}(x,t)$ and $f^{-}(x,t)$, with $x\in \Omega\subset \mathds{R}$ and $t\in \mathds{R}^{+}$, where $f^{+}(x,t)\,$ (respectively $f^{-}(x,t)$) denotes the density of particles at time $t>0$ travelling  along a straight line with velocity $+c$ (respectively $-c$). The particle changes with rate $\sigma$ the direction. The differential model can be written as
\begin{equation*}
    \begin{aligned}
        f^{+}_{t} + c f^{+}_{x} & = \sigma (f^{-} - f^{+}), \\
        f^{-}_{t} - c f^{-}_{x} & = \sigma (f^{+} - f^{-}).
    \end{aligned}
\end{equation*}
Introducing the macroscopic variables $\rho = f^{+} + f^{-},\quad j = c (f^{+} - f^{-}),$
we obtain the equivalent form
\begin{equation*}
\left\{\;
    \begin{aligned}
        & \partial_{t}\rho + \partial_{x} j  = 0,\\
        & \partial_{t} j + c^{2}\partial_{x}\rho  = 2\sigma j.
    \end{aligned}
\right.
\end{equation*}
Setting $c^{2} = 2\sigma = 1/\varepsilon$ where  $\varepsilon>0$ is the relaxation parameter, yields
\begin{equation}\label{eq: pde GT}
\left\{\;
    \begin{aligned}
        & \partial_{t} \rho = - \partial_{x} j,  \\
        & \partial_{t}j = \frac{1}{\varepsilon}\left( - j -\partial_{x}\rho \right),
    \end{aligned}
\right.
\end{equation}
which has the desired singularly perturbed structure. When $\varepsilon\to 0$, we obtain
\begin{equation}\label{eq: pde GT 2}
    j(x,t) = - \partial_{x}\rho(x,t).
\end{equation}
The previous results can be applied after using a semi--discretisation in space. This leads to a similar structure as in the previous section. Given two controls $\alpha,\beta$ (not necessarily different) and two matrices $\mathds{H},\mathds{G}$ depending on $\bm{\rho}$, the singularly perturbed discrete dynamics is
\begin{equation}\label{eq: relax rho j}
\left\{\;
    \begin{aligned}
        & \dot{\bm{\rho}}_{\varepsilon}(s) =  -D\bm{J}_{\varepsilon}(s) + \mathds{H}(\bm{\rho}_{\varepsilon}(s))\alpha(s),  && \bm{\rho}_{\varepsilon}(0) = \bm{\rho}_{0} \in \mathds{R}^{m},\\
        &\dot{\bm{J}}_{\varepsilon}(s) = - \frac{1}{\varepsilon}[D \bm{\rho}_{\varepsilon}(s) + \bm{J}_{\varepsilon}(s) - \mathds{G}(\bm{\rho}_{\varepsilon}(s))\beta(s)], && \bm{J}_{\varepsilon}(0) = \bm{J}_{0}\in \mathds{R}^{m},
    \end{aligned}
\right.
\end{equation}
and the corresponding reduced dynamics is
\begin{equation}\label{eq: rho j}
        \dot{\bm{\rho}}(s) + D [-D\bm{\rho}(s) + \mathds{G}(\bm{\rho}_{\varepsilon}(s))\beta(s)] = \mathds{H}(\bm{\rho}_{\varepsilon}(s))\alpha(s) ,  \quad \bm{\rho}(0) = \bm{\rho}_{0} \in \mathds{R}^{m}.
\end{equation}
The latter dynamics corresponds to the discretisation of the controlled PDE
\begin{equation*}
\left\{\;
\begin{aligned}
    & \partial_{t}\rho(t,x) + \partial_{x}[-\partial_{x}\rho(t,x) + \mathds{G}(\rho(t,x))\beta(t,x) ]= \mathds{H}(\rho(t,x))\alpha(t,x),\\
    & \rho(0,x) = \rho_{0}(x), \quad (t,x)\in (0,+\infty)\times \mathds{R}.
\end{aligned}
\right.
\end{equation*}
In the absence of controls, this becomes the heat equation
\begin{equation}\label{eq: heat}
    \partial_{t}\rho(t,x) = \partial_{xx}\rho(t,x), \quad \rho(0,x) = \rho_{0}(x),\quad (t,x)\in (0,+\infty)\times \mathds{R} ,
\end{equation}
which is also what one gets when substituting \eqref{eq: pde GT 2} in the first PDE of \eqref{eq: pde GT}. 

Let us introduce the two optimal control problems
\begin{equation}\label{eq: value eps GT}
\tag{SP.3}
    \inf \; \Phi(\bm{\rho}_{\varepsilon}(T)),\quad \text{s.t. } \; \eqref{eq: relax rho j}\,,
\end{equation}
and
\begin{equation}\label{eq: value eff GT}
\tag{R.3}
    \inf \; \Phi(\bm{\rho}_{\varepsilon}(T)),\quad \text{s.t. } \; \eqref{eq: rho j}\,.
\end{equation}
The following corollary is a direct consequence of Theorem \ref{thm: conv G}.

\begin{cor}\label{cor: GT}
    As $\varepsilon\to 0$, the problem \eqref{eq: value eps GT} is approximated by \eqref{eq: value eff GT} in the sense of Definition \ref{def:approx}. In particular, the controlled system \eqref{eq: relax rho j} with any given initial conditions converges to \eqref{eq: rho j} locally uniformly on any finite time interval, and we have $\inf\, \Phi(\bm{\rho}_{\varepsilon}(T)) \to \inf \, \Phi(\bm{\rho}(T))$. 
\end{cor} 

\subsection{Goldstein-Taylor three--scale model}\label{sec: GT 3}

Recalling the model in the previous section \S \ref{sec: GT 2}, we would like now to get a three--scale approximation. To do so, we suppose we have an additional term $\mathcal{F}_{1}(\rho)$ of order $\varepsilon$ in \eqref{eq: heat} 
\begin{equation*}
    \partial_{t}\rho(t,x) + \partial_{x}\big[-\partial_{x}\rho(t,x) + \varepsilon \, \mathcal{F}_{1}(\rho(t,x))\big] = 0.
\end{equation*}
The three--scale relaxation becomes, for some $a,b>0$ fixed,
\begin{equation*}
\left\{\;
\begin{aligned}
    & \partial_{t}\rho = -\partial_{x}j,\\
    & \partial_{t} j  + a\, \partial_{x} \rho 
    = - \frac{1}{\varepsilon}\big[ j - ( -\partial_{x}\rho +\varepsilon w) \big],\\
    & \partial_{t} w + b\, \partial_{x}\rho 
    = -\frac{1}{\varepsilon^{2}}\big[ w - \mathcal{F}_{1}(\rho) \big].
\end{aligned}
\right.
\end{equation*}
Repeating what we have done in the beginning of \S \ref{sec: JX 3 app}, we get
\begin{equation}\label{with ab}
\left\{\;
\begin{aligned}
    & \partial_{t} \rho &=& -\partial_{x}[j-p] - \partial_{x}p,\\
    & \partial_{t}[j-p] &=& \; - a \, \partial_{x}\rho +
    [w-q] + q,\\
    & \partial_{t} [w-q] &=&\; - b\, \partial_{x}\rho,
    \\
    & \partial_{t} p &=& -\frac{1}{\varepsilon}\big( [j-p] +\partial_{x}\rho + p\big),\\
    & \partial_{t} q &=& - \frac{1}{\varepsilon^{2}}\big([w-q] - \mathcal{F}_{1}(\rho) + q\big).
\end{aligned}
\right.
\end{equation}
Note if we choose $a=b=0$, the latter system simplifies as follows
\begin{equation}\label{without ab}
\left\{\;
\begin{aligned}
    & \partial_{t} \rho &=& -\partial_{x}[j-p] - \partial_{x}p,\\
    & \partial_{t}[j-p] &=& \; w,\\
    & \partial_{t} p &=& -\frac{1}{\varepsilon}\big( [j-p] +\partial_{x}\rho + p\big),\\
    & \partial_{t} w &=& - \frac{1}{\varepsilon^{2}}\big(w - \mathcal{F}_{1}(\rho) \big).
\end{aligned}
\right.
\end{equation}
For the sake of generality, we let $a,b\geq 0$. In particular, they are allowed to be null, contrary to Jin--Xin relaxation in the previous sections.

Let us consider the semi-discretisation in space of the latter PDE system by introducing 
$(\bm{\rho}_{\varepsilon}, \bm{J}_{\varepsilon}, \bm{w}_{\varepsilon}, \bm{p}_{\varepsilon}, \bm{q}_{\varepsilon})$ the discretisation of $(\rho, [j-p], [w-q], p, q)$ respectively, and letting $D$ be for example a first order finite-volume spatial discretisation of the transport operator $\partial_{x}$. We shall also introduce three control parameters (not necessarily different) $\alpha,\beta,\gamma$ and three matrices $\mathds{H},\mathds{G},\mathds{K}$ functions of $\bm{\rho}_{\varepsilon}$. Then one gets the system of controlled and singularly perturbed ODEs with $s\in [0,1]$,
\begin{equation}\label{relax 3 eps}
\left\{\;
\begin{aligned}
    \dot{\bm{\rho}}_{\varepsilon}(s) & = -D \bm{J}_{\varepsilon}(s) - D\bm{p}_{\varepsilon}(s)  + \mathds{H}(\bm{\rho}_{\varepsilon}(s))\alpha(s), && \bm{\rho}_{\varepsilon}(0) = \bm{\rho}_{0}\in \mathds{R}^{m},\\
    \dot{\bm{J}}_{\varepsilon}(s) & =\; -a \, D\bm{\rho}_{\varepsilon}(s) + 
    \bm{w}_{\varepsilon}(s) + \bm{q}_{\varepsilon}(s), && \bm{J}_{\varepsilon}(0) = \bm{J}_{0}\in \mathds{R}^{m},\\
    \dot{\bm{w}}_{\varepsilon}(s) & =\; -b\, D \bm{\rho}_{\varepsilon}(s), && \bm{w}_{\varepsilon}(0) = \bm{w}_{0}\in \mathds{R}^{m},\\
    \dot{\bm{p}}_{\varepsilon}(s) & = -\frac{1}{\varepsilon}\big[\bm{p}_{\varepsilon}(s) + \bm{J}_{\varepsilon}(s) + D \bm{\rho}_{\varepsilon}(s)  - \mathds{G}(\bm{\rho}_{\varepsilon}(s))\beta(s)\big], && \bm{p}_{\varepsilon}(0) = \bm{p}_{0}\in \mathds{R}^{m},\\
    \dot{\bm{q}}_{\varepsilon}(s) & = -\frac{1}{\varepsilon^{2}} \big[\bm{q}_{\varepsilon}(s) +  \bm{w}_{\varepsilon}(s) - \mathcal{F}_{1}(\bm{\rho}_{\varepsilon}(s))  - \mathds{K}(\bm{\rho}_{\varepsilon}(s))\gamma(s)\big], &&\bm{q}_{\varepsilon}(0)= \bm{q}_{0} \in \mathds{R}^{m}.
\end{aligned}
\right.
\end{equation}
For simplicity, we shall consider a control problem whose running cost and final costs are functions of $\bm{\rho}_{\varepsilon}$ only, the general case being discussed in Remark \ref{rem: cost general},
\begin{equation}\label{SP GT multi}
\tag{SP.4}
    \mathbf{V}^{\varepsilon}(\bm{\rho}_{0}, \bm{J}_{0}, \bm{w}_{0}, \bm{p}_{0}, \bm{q}_{0}) = 
    \inf \; \int_{0}^{1} \ell(\bm{\rho}_{\varepsilon}(s))\,\text{d}s +  \phi(z(1)), \; \text{ s.t.: } \;  \eqref{relax 3 eps}.
\end{equation}
Its corresponding reduced problem is
\begin{equation}\label{SP GT multi R}
\tag{R.4}
    \mathbf{V}_{0}(\bm{\rho}_{0}) =
    \inf \; \int_{0}^{1} \ell(\bm{\rho}(s))\,\text{d}s +  \phi(z(1)), \; \text{ s.t.: } \;  \eqref{eq: rho j}.
\end{equation}
A consequence of Theorem \ref{thm: conv multi} is the following.

\begin{cor}\label{cor GT 3}
    The value function $\mathbf{V}^{\varepsilon}$ of \eqref{SP GT multi} converges locally uniformly on $\big(\mathds{R}^{m}\big)^{5}\times (0,1]$ to the value function $\mathbf{V}_{0}$ of \eqref{SP GT multi R}.
\end{cor}

\subsection{Shallow water and inviscid Burger's equation}

Let us recall the system as described in \cite[\S 6.1]{pareschi2005implicit} (see also \cite{jin1995runge}) for shallow water flow
\begin{equation*}
\left\{\;
    \begin{aligned}
        & \partial_{t} h + \partial_{x}(hv) = 0,\\
        & \partial_{t}(hv) + \partial_{x}\left(h + \frac{1}{2}h^{2}\right) = \frac{1}{\varepsilon}h\left(\frac{h}{2} - v\right),
    \end{aligned}
\right.
\end{equation*}
where $h$ is the height of the water with respect to the bottom and $hv$ the flux. 
The zero relaxation limit of this model is given by the inviscid Burgers equation.

Choosing $g$ such that $\partial_{t} g = \frac{1}{\varepsilon}\left(\frac{h^{2}}{2} - hv\right)$ and letting $f:= hv-g$ yield the system
\begin{equation*}
\left\{\;
    \begin{aligned}
        & \partial_{t} h = -  \partial_{x} f  -\partial_{x} g,\\
        & \partial_{t} f = - \partial_{x}\left(h + \frac{h^{2}}{2}\right), \\
        & \partial_{t} g = -\frac{1}{\varepsilon}\left(g+f -\frac{h^{2}}{2}\right).
    \end{aligned}
\right.
\end{equation*}
Hence, we recover the desired singularly perturbed structure. In the finite dimensional case, it becomes after introducing two control parameters $\alpha,\beta$
\begin{equation}\label{eq: relax WB}
\left\{\;
\begin{aligned}
    & \dot{\bm{h}}_{\varepsilon}(s) =  -D \bm{f}_{\varepsilon}(s) - D \bm{w}_{\varepsilon}(s) + \mathds{H}(\bm{h}_{\varepsilon}(s))\alpha(s), && \bm{h}_{\varepsilon}(0) = \bm{h}_{0} \in \mathds{R}^{m},\\
    & \dot{\bm{f}}_{\varepsilon}(s) = -D \bm{h}_{\varepsilon}(s) - \frac{1}{2}D (\bm{h}_{\varepsilon}(s))^{2}, && \bm{f}_{\varepsilon}(0) = \bm{f}_{0} \in \mathds{R}^{m},\\
    & \dot{\bm{g}}_{\varepsilon}(s) = -\frac{1}{\varepsilon}\left[\bm{g}_{\varepsilon}(s) + \bm{f}_{\varepsilon}(s) - \frac{1}{2}(\bm{h}_{\varepsilon}(s))^{2} - \mathds{G}(\bm{h}_{\varepsilon}(s))\beta(s)\right], && \bm{w}_{\varepsilon}(0) = \bm{w}_{0}\in \mathds{R}^{m}.
\end{aligned}
\right.
\end{equation}
Here, $((\bm{h}_{\varepsilon}(s))^{2})$ is the vector whose entries are the square of the entries of $\bm{h}_{\varepsilon}(s)$. The control parameters can for example refer to the bottom profile. At the limit $\varepsilon \to 0$, one gets 
\begin{equation}\label{eq: eff WB}
    \dot{\bm{h}}(s) =  -\frac{1}{2}D(\bm{h}(s))^{2} - D\mathds{G}(\bm{h}(s))\beta(s) + \mathds{H}(\bm{h}(s))\alpha(s),
\end{equation}
which is as expected the discretisation of the controlled inviscid Burgers equation
\begin{equation*}
    \partial_{t}h(t,x) + \frac{1}{2}\partial_{x}\big[h^{2}(t,x) + \mathds{G}(h(t,x))\beta(t,x)\big] = \mathds{H}(h(t,x))\alpha(t,x).
\end{equation*}

Let us introduce the two optimal control problems
\begin{equation}\label{eq: value eps WB}
\tag{SP.5}
    \inf \; \Phi(\bm{h}_{\varepsilon}(T)),\quad \text{s.t. } \; \eqref{eq: relax WB}\,,
\end{equation}
and
\begin{equation}\label{eq: value eff WB}
\tag{R.5}
    \inf \; \Phi(\bm{h}_{\varepsilon}(T)),\quad \text{s.t. } \; \eqref{eq: eff WB}\,.
\end{equation}
The following corollary is a direct consequence of Theorem \ref{thm: conv G}.

\begin{cor}\label{cor: WB}
    As $\varepsilon\to 0$, the problem \eqref{eq: value eps WB} is approximated by \eqref{eq: value eff WB} in the sense of Definition \ref{def:approx}. In particular, the controlled system \eqref{eq: relax WB} with any given initial conditions converges to \eqref{eq: eff WB} locally uniformly on any finite time interval, and we have $\inf\, \Phi(\bm{h}_{\varepsilon}(T)) \to \inf \, \Phi(\bm{h}(T))$. 
\end{cor}

The three--scale approximation follows the computations in \S \ref{sec: JX 3 app} and in \S \ref{sec: GT 3}. Starting from 
\begin{equation*}
\left\{\;
\begin{aligned}
    & \partial_{t}h + \partial_{x}\left[ hv + \varepsilon \mathcal{F}_{1}(h) \right] = 0,\\
    & \partial_{t}(hv) + \partial_{x}\left(h + \frac{1}{2}h^{2}\right) = \frac{1}{\varepsilon}\left(\frac{1}{2}h^{2}-hv\right),
\end{aligned}
\right.
\end{equation*}
we set \; $\partial_{t}g = \frac{1}{\varepsilon}\left(\frac{1}{2}h^{2} - hv\right),$\,
which yields $\partial_{t}[hv -g] = -\partial_{x}\left(h+\frac{1}{2}h^{2}\right).$
We define $f:= hv - g$ and get the following
\begin{equation*}
\left\{\;
\begin{aligned}
    & \partial_{t} h =- \partial_{x}(f+g + \varepsilon\mathcal{F}_{1}(h)), \\
    & \partial_{t} f = -\partial_{x}\left(h + \frac{1}{2}h^{2}\right),\\
    & \partial_{t} g = -\frac{1}{\varepsilon}\left( g+f -\frac{1}{2}h^{2}\right).
\end{aligned}
\right.
\end{equation*}
Let us introduce $p$ and $q$ such that
\begin{equation*}
\begin{aligned}
    & \partial_{t} h = -\partial_{x} k,\\
    & \partial_{t} k = -\frac{1}{\varepsilon}\big[ k- (f+g+\varepsilon q) \big], \\
    & \partial_{t} p = - \frac{1}{\varepsilon}\big[ k - (f+g) \big],\\
    & \partial_{t} q = -\frac{1}{\varepsilon^{2}}[q - \mathcal{F}_{1}(h)].
\end{aligned}
\end{equation*}
This yields
\begin{equation*}
\left\{\;
\begin{aligned}
    & \partial_{t} h = -\partial_{x}[k-p] - \partial_{x}p,\\
    & \partial_{t} [k-p] =  q,\\
    & \partial_{t} p = - \frac{1}{\varepsilon}\big[ p + [k-p] - (f+g) \big],\\
    & \partial_{t} q = -\frac{1}{\varepsilon^{2}}[q - \mathcal{F}_{1}(h)],
\end{aligned}
\right.
\end{equation*}
to which we add the two equations for $f,g$
\begin{equation*}
    \partial_{t} f = -\partial_{x}\left(h + \frac{1}{2}h^{2}\right), \quad \text{ and } \quad \partial_{t} g = -\frac{1}{\varepsilon}\left( g+f -\frac{1}{2}h^{2}\right).
\end{equation*}
Note here we did not add the terms in Jin--Xin relaxation (i.e. we took $a=b=0$ in \eqref{relax JX 2}). Therefore we got a system analogue to \eqref{without ab}. But we could also consider the additional terms with $a,b>0$, and get a system analogue to \eqref{with ab}.

The discrete version can be expressed as follows. Let $(\bm{h}, \bm{f}, \bm{g}, \bm{k}, \bm{p}, \bm{q})$ be the discretisation of $(h, f, g, [k-p], p, q)$ respectively. Then one gets

\begin{equation}\label{SP WB multi}
\left\{\;
\begin{aligned}
    & \dot{\bm{h}}_{\varepsilon}(s) = - D \bm{k}_{\varepsilon}(s) - D \bm{p}_{\varepsilon}(s) + \mathds{H}(\bm{h}_{\varepsilon}(s))\alpha(s) , &&  \, \bm{h}_{\varepsilon}(0) = \bm{h}_{0},\\
    & \dot{\bm{f}}_{\varepsilon}(s) = - D\bm{h}_{\varepsilon}(s) - \frac{1}{2}D (\bm{h}_{\varepsilon}(s))^{2} , &&  \, \bm{f}_{\varepsilon}(0) = \bm{f}_{0},\\
    & \dot{\bm{k}}_{\varepsilon}(s) = \bm{q}_{\varepsilon}(s), &&  \, \bm{k}_{\varepsilon}(0) = \bm{k}_{0},\\
    & \dot{\bm{g}}_{\varepsilon}(s) =  -\frac{1}{\varepsilon}\left[ \bm{g}_{\varepsilon}(s) + \bm{f}_{\varepsilon}(s) - \frac{1}{2}(\bm{h}_{\varepsilon}(s))^{2} -\mathds{G}(\bm{h}_{\varepsilon}(s))\beta(s) \right], &&  \, \bm{g}_{\varepsilon}(0) = \bm{g}_{0},\\
    & \dot{\bm{p}}_{\varepsilon}(s) = -\frac{1}{\varepsilon}\left[ \bm{p}_{\varepsilon}(s) + \bm{k}_{\varepsilon}(s) - \bm{f}_{\varepsilon}(s) - \bm{g}_{\varepsilon}(s) \right], &&  \, \bm{p}_{\varepsilon}(0) = \bm{p}_{0},\\
    & \dot{\bm{q}}_{\varepsilon}(s) = -\frac{1}{\varepsilon^{2}}\left[\bm{q}_{\varepsilon}(s) - \mathcal{F}_{1}(\bm{h}_{\varepsilon}(s)) - \mathds{K}(\bm{h}_{\varepsilon}(s))\alpha(s)\right], &&  \, \bm{q}_{\varepsilon}(0) = \bm{q}_{0},
\end{aligned}
\right.
\end{equation}
where we have chosen to add three control parameters $\alpha,\beta,\gamma$, although all the ODEs can be controlled as described in section \S \ref{sec: multiscale sys}. At the limit $\varepsilon\to 0$, we obtain \eqref{eq: eff WB}. This is the object of the following result, analogue of Corollary \ref{cor JX 3}.
\begin{cor}
    The value function corresponding to the control of \eqref{SP WB multi} converges locally uniformly on $\big(\mathds{R}^{m}\big)^{6}\times (0,1]$ to the value function corresponding to the control of \eqref{eq: eff WB}.
\end{cor}

\subsection{Second-order traffic flow models}

As stated in \cite[\S 6.2]{pareschi2005implicit}, we recall the model in \cite{aw2000resurrection} for vehicular traffic
\begin{equation*}
\left\{\;
    \begin{aligned}
        & \partial_{t} \rho + \partial_{x}(\rho v) = 0,\\
        & \partial_{t}(\rho \omega) + \partial_{x}(v\rho \omega) = A \frac{\rho}{\varepsilon}\big(V(\rho) - v\big),\\
        & \omega = v - P(\rho),
    \end{aligned}
\right.
\end{equation*}
where $\rho$ is the density of vehicles subject to the first (continuity) equation, and complemented with an additional velocity equation for the mass flux variations due to the road conditions in front of the driver. Here $P(\rho)$ is a given function describing the anticipation of road conditions in front of the drivers and $V(\rho)$ describes the dependence of the velocity with respect to the density for an equilibrium situation. The parameter $\varepsilon$ is the relaxation time and $A>0$ is a positive constant. \\
When the relaxation time goes to zero, the Lighthill-Whitham \cite{whitham1974linear} model is obtained
\begin{equation}\label{eq: LW}
        \partial_{t} \rho + \partial_{x}\big( \rho V(\rho) \big) = 0.
\end{equation}

We introduce a new variable $f$ subject to the PDE $\;\partial_{t} f = A\frac{\rho}{\varepsilon}\big(V(\rho) - v\big),$
and define $g := \rho \omega -f$. Then we obtain the system
\begin{equation}\label{eq: traffic flow SP}
\left\{\;
    \begin{aligned}
        & \partial_{t} \rho + v(\partial_{x}\rho) + \rho (\partial_{x} v) = 0, \\
        & \partial_{t} g + (\partial_{x}v) g + v (\partial_{x} g) + (\partial_{x}v)f + v(\partial_{x}f) = 0,\\
        & \partial_{t} f = A\frac{1}{\varepsilon}\big(\rho V(\rho) + \rho P(\rho) - f - g \big),
    \end{aligned}
\right.
\end{equation}
and $v$ solves the equation $\rho v = f+g - \rho P(\rho)$. 

The model \eqref{eq: traffic flow SP} falls within our setting, and its discretised version is the following system of ODEs to which we added two control parameters $\alpha$ and $\beta$
\begin{equation}\label{TF eps}
\left\{\;
    \begin{aligned}
        \dot{\bm{\rho}}_{\varepsilon}(s) & = -\bm{v}_{\varepsilon}(s) D\bm{\rho}_{\varepsilon}(s) - \bm{\rho}_{\varepsilon}(s) D\bm{v}_{\varepsilon}(s) + \mathds{H}(\bm{\rho}_{\varepsilon}(s))\alpha(s),\\
        \dot{\bm{g}}_{\varepsilon}(s) & = -\big(\bm{f}_{\varepsilon}(s) + \bm{g}_{\varepsilon}(s)\big) D\bm{v}_{\varepsilon}(s) -\bm{v}_{\varepsilon}(s) D\big(\bm{f}_{\varepsilon}(s) +\bm{g}_{\varepsilon}(s)\big), \\
        \dot{\bm{f}}_{\varepsilon}(s) & = -\frac{1}{\varepsilon}A\big[ \bm{f}_{\varepsilon}(s) + \bm{g}_{\varepsilon}(s) - \bm{\rho}_{\varepsilon}(s) V(\bm{\rho}_{\varepsilon}(s)) - \bm{\rho}_{\varepsilon}(s) P(\bm{\rho}_{\varepsilon}(s))-\mathds{G}(\bm{\rho}_{\varepsilon}(s))\beta(s)\big],
    \end{aligned}
\right.
\end{equation}
together with the equation defining $\bm{v}_{\varepsilon}(s)$
\begin{equation}\label{TF v}
    \bm{\rho}_{\varepsilon}(s) \bm{v}_{\varepsilon}(s) = \bm{f}_{\varepsilon}(s) + \bm{g}_{\varepsilon}(s) - \bm{\rho}_{\varepsilon}(s)P(\bm{\rho}_{\varepsilon}(s)). 
\end{equation}
Here the multiplication of two vectors (e.g., in $\bm{\rho}_{\varepsilon}(s) \bm{v}_{\varepsilon}(s)$) is understood component-wise. As before, we denoted by $D$ a discretisation matrix. An example of control could also be $V(\bm{\rho}_{\varepsilon}) = \beta$. In this case, the dynamics for $\bm{f}_{\varepsilon}$ becomes
\begin{equation*}
    \dot{\bm{f}}_{\varepsilon}(s)  = -\frac{1}{\varepsilon}A\big[ \bm{f}_{\varepsilon}(s) + \bm{g}_{\varepsilon}(s) - \bm{\rho}_{\varepsilon}(s) \beta(s) - \bm{\rho}_{\varepsilon}(s) P(\bm{\rho}_{\varepsilon}(s))\big].
\end{equation*}
When $\varepsilon\to 0$, the local equilibrium is
\begin{equation*}
    \bm{f}(s) + \bm{g}(s) = \bm{\rho}(s) \big[V(\bm{\rho}(s)) +  P(\bm{\rho}(s))\big]+\mathds{G}(\bm{\rho}(s))\beta(s).
\end{equation*}
When substituted in \eqref{TF v}, one gets
\begin{equation*}
\begin{aligned}
    \bm{\rho}(s) \bm{v}(s) = \bm{\rho}(s) V(\bm{\rho}(s)) +\mathds{G}(\bm{\rho}(s))\beta(s).
\end{aligned}
\end{equation*}
We can choose the matrix $\mathds{G}(\bm{\rho}(s))$ to be zero matrix when $\bm{\rho}(s) =0$. Then, we can write, whenever $\bm{\rho}(s) \neq 0$
\begin{equation*}
\begin{aligned}
    \bm{v}(s) =  V(\bm{\rho}(s)) +(\bm{\rho}(s))^{-1}\mathds{G}(\bm{\rho}(s))\beta(s),
\end{aligned}
\end{equation*}
the inverse $(\bm{\rho}(s))^{-1}$ being understood component-wise. Therefore, the first equation in \eqref{TF eps} becomes
\begin{equation}\label{eq: rho TF}
\begin{aligned}
    \dot{\bm{\rho}}(s) = & -V(\bm{\rho}(s)) D\bm{\rho}(s) - \bm{\rho}(s) DV(\bm{\rho}(s)), \\
    & \quad  -\left[(\bm{\rho}(s))^{-1}\mathds{G}(\bm{\rho}(s))\beta(s)\right] D\bm{\rho}(s) - \bm{\rho}(s) D\left[(\bm{\rho}(s))^{-1}\mathds{G}(\bm{\rho}(s))\beta(s)\right],\\
    & \quad \quad \quad \quad \quad + \mathds{H}(\bm{\rho}(s))\alpha(s).
\end{aligned}
\end{equation}
This corresponds to the discretisation of the controlled PDE
\begin{equation*}
    \partial_{t}\rho + \partial_{x} \big[\rho(t,x)V(\rho(t,x)) + \mathds{G}(\rho(t,x))\beta(t,x)\big] = \mathds{H}(\rho(s))\alpha(t,x),
\end{equation*}
which reduces to \eqref{eq: LW} in the absence of controls.

Let us introduce the two optimal control problems
\begin{equation}\label{eq: value eps TF}
\tag{SP.6}
    \inf \; \Phi(\bm{\rho}_{\varepsilon}(T)),\quad \text{s.t. } \; \eqref{TF eps}-\eqref{TF v}\,,
\end{equation}
and
\begin{equation}\label{eq: value eff TF}
\tag{R.6}
    \inf \; \Phi(\bm{\rho}_{\varepsilon}(T)),\quad \text{s.t. } \; \eqref{eq: rho TF}\,.
\end{equation}
The following corollary is a direct consequence of Theorem \ref{thm: conv G}.

\begin{cor}
    As $\varepsilon\to 0$, the problem \eqref{eq: value eps TF} is approximated by \eqref{eq: value eff TF} in the sense of Definition \ref{def:approx}. In particular, the controlled system \eqref{TF eps}-\eqref{TF v} with any given initial conditions converges to \eqref{eq: rho TF} locally uniformly on any finite time interval, and we have $\inf\, \Phi(\bm{\rho}_{\varepsilon}(T)) \to \inf \, \Phi(\bm{\rho}(T))$. 
\end{cor}

\subsection{Granular Gases}

We shall consider the model in \cite[\S 6.3]{pareschi2005implicit} for a granular gas \cites{jenkins1985grad, toscani2004kinetic} given by
\begin{equation*}
\left\{ \;
    \begin{aligned}
        & \partial_{t}\rho + \partial_{x}(\rho u) = 0, \\
        & \partial_{t}(\rho u) + \partial_{x}(\rho u^2 + p) = \rho g, \\
        & \partial_{t}\left( \frac{1}{2}\rho u^2 + \frac{3}{2}\rho T \right) + \partial_{x}\left( \frac{1}{2}\rho u^3 + \frac{3}{2} u\rho T + pu \right) = - \frac{(1-e^2)}{\varepsilon} G(\rho) \rho^{2} T^{3/2},
    \end{aligned}
\right.
\end{equation*}
where $e$ is the coefficient of restitution, $g$ the acceleration due to gravity, $\varepsilon$ a relaxation time, $p$ is the pressure given by $p = \rho T (1+2(1+e)G(\rho)),$
and $G(\rho)$ is the statistical correlation function. 

We introduce new variables
\begin{equation}\label{eq: v w}
    v=\rho u, \quad \text{ and } \quad \text{w} =  \frac{1}{2}\rho u^2 + \frac{3}{2}\rho T,
\end{equation}
and define $\varphi$ such that $\partial_{t}\varphi = - \frac{(1-e^2)}{\varepsilon} G(\rho) \rho^{2} T^{3/2}.$ 
Then, setting
\begin{equation}\label{eq: omega}
    \omega = \text{w} - \varphi,
\end{equation}
yields the system 
\begin{equation*}
\left\{ \;
    \begin{aligned}
        & \partial_{t}\rho + \partial_{x}v = 0,\\
        & \partial_{t} v + \partial_{x}(uv + p) = \rho g,\\
        & \partial_{t} \omega + \partial_{x}(u \omega + up +u\varphi) = 0,\\
        & \partial_{t}\varphi = - \frac{(1-e^2)}{\varepsilon} G(\rho) \rho^{2} T^{3/2}.
    \end{aligned}
\right.
\end{equation*}
From the definition of $\omega$ and $\text{w}$ above, we can write $\rho T = \frac{2}{3}(\omega + \varphi) - \frac{1}{3}v u$, then we substitute the latter quantity in $\rho^{2}T^{3/2} = (\rho T)\rho T^{1/2}$ and obtain the system
\begin{equation*}
\left\{ \;
    \begin{aligned}
        & \partial_{t}\rho + \partial_{x}v = 0,\\
        & \partial_{t} v + \partial_{x}(uv + p) = \rho g,\\
        & \partial_{t} \omega + \partial_{x}(u \omega + up +u\varphi) = 0,\\
        & \partial_{t}\varphi = \frac{1}{\varepsilon}\,\frac{2(1-e^2)}{3} G(\rho) \rho T^{1/2}\left( -\omega - \varphi + \frac{1}{2}vu\right),
    \end{aligned}
\right.
\end{equation*}
whose structure now complies with our singularly perturbed model. 

Indeed, the discrete and controlled version of the latter system of PDEs is the following system of ODEs where the control parameters are $\alpha,\beta,\gamma$
\begin{equation}\label{eq: granular relax}
\left\{\,
\begin{aligned}
    \dot{\bm{\rho}}_{\varepsilon}(s) & = -D \bm{v}_{\varepsilon}(s) + \mathds{H}(\bm{\rho}_{\varepsilon}(s))\alpha(s), && \bm{\rho}_{\varepsilon}(0) = \bm{\rho}_{0},\\
    \dot{\bm{v}}_{\varepsilon}(s) & = \bm{\rho}_{\varepsilon}(s)g - \bm{u}_{\varepsilon}(s)D \bm{v}_{\varepsilon}(s) - \bm{v}_{\varepsilon}(s)D\bm{u}_{\varepsilon}(s) - D \bm{p}_{\varepsilon}(s)+ \mathds{K}(\bm{\rho}_{\varepsilon}(s))\gamma(s),&& \bm{v}_{\varepsilon}(0) = \bm{v}_{0},\\
    \dot{\bm{w}}_{\varepsilon}(s) & = -\bm{u}_{\varepsilon}(s) D\big[\bm{w}_{\varepsilon}(s) + \bm{p}_{\varepsilon}(s) + \bm{\varphi}_{\varepsilon}(s)\big]  && \bm{w}_{\varepsilon}(0) = \bm{w}_{0},\\
    & \quad \quad \quad - \big[\bm{w}_{\varepsilon}(s) + \bm{p}_{\varepsilon}(s) + \bm{\varphi}_{\varepsilon}(s)\big] D \bm{u}_{\varepsilon}(s), && \\
    \dot{\bm{\varphi}}_{\varepsilon}(s) & = -\frac{1}{\varepsilon} M(\bm{\rho}_{\varepsilon}(s))\left[ \bm{\varphi}_{\varepsilon}(s) + \bm{w}_{\varepsilon}(s) - \frac{1}{2}\bm{u}_{\varepsilon}(s)\bm{v}_{\varepsilon}(s) - \mathds{G}(\bm{\rho}_{\varepsilon}(s))\beta(s)\right] , && \bm{\varphi}_{\varepsilon}(0) = \bm{\varphi}_{0},
\end{aligned}
\right.
\end{equation}
where $M(\bm{\rho}_{\varepsilon}(s)):=\frac{2(1-e^2)}{3} G(\bm{\rho}_{\varepsilon}(s) )\bm{\rho}_{\varepsilon}(s) T^{1/2}$. 

When $\varepsilon\to 0$, we obtain 
\begin{equation*}
    \bm{\varphi}(s) + \bm{w}(s) = \frac{1}{2}\bm{u}(s)\bm{v}(s) + \mathds{G}(\bm{\rho}(s))\beta(s).
\end{equation*}
In particular since $\dot{\bm{\varphi}}(s) = 0$ and recalling \eqref{eq: omega}, we get
\begin{equation*}
    \dot{\textbf{w}}(s) = \dot{\bm{w}}(s) + \dot{\bm{\varphi}}(s) = \dot{\bm{w}}(s). 
\end{equation*}
Thus, from \eqref{eq: v w}, we have
\begin{equation}\label{eq: granular 1}
\begin{aligned}
    \frac{\text{d}}{\text{d}t} \left[\frac{1}{2} \bm{\rho}(s)\bm{u}(s)^{2} + \frac{3}{2}\bm{\rho}(s)T\right] 
    & = -\bm{u}(s) D\big[\frac{1}{2}\bm{u}(s)\bm{v}(s) + \mathds{G}(\bm{\rho}(s))\beta(s) + \bm{p}(s)\big]  \\
    & \quad \quad \quad - \big[\frac{1}{2}\bm{u}(s)\bm{v}(s) + \mathds{G}(\bm{\rho}(s))\beta(s) + \bm{p}(s)\big] D \bm{u}(s), 
\end{aligned}
\end{equation}
which complements
\begin{equation}\label{eq: granular 2}
\left\{\,
\begin{aligned}
    \dot{\bm{\rho}}(s) & = -D \bm{v}(s) + \mathds{H}(\bm{\rho}(s))\alpha(s), \\
    \dot{\bm{v}}(s) & = \bm{\rho}(s)g - \bm{u}(s)D \bm{v}(s) - \bm{v}(s)D\bm{u}(s) - D \bm{p}(s) + \mathds{K}(\bm{\rho}(s))\gamma(s). 
\end{aligned}
\right.
\end{equation}
Recalling again \eqref{eq: v w}, the latter system of ODEs corresponds to the discretisation of the system of PDEs
\begin{equation*}
\left\{\,
\begin{aligned}
    & \partial_{t} \rho + \partial_{x}\big[ \rho u \big] = \mathds{H}(\rho)\alpha,\\
    & \partial_{t}\big(\rho u\big) + \partial_{x}\big(u^{2}\rho + p\big) = \rho g + \mathds{K}(\rho)\gamma,\\
    & \partial_{t}\left(\frac{1}{2}\rho u^{2} + \frac{3}{2}\rho T\right) +\partial_{x}\left[ \frac{1}{2}u^{3}\rho  + up + u\,\mathds{G}(\rho)\beta\right] = 0.
\end{aligned}
\right.
\end{equation*}

Let us introduce the two optimal control problems
\begin{equation}\label{eq: value eps GG}
\tag{SP.6}
    \inf \; \Phi(\bm{\rho}_{\varepsilon}(T)),\quad \text{s.t. } \; \eqref{eq: granular relax}\,,
\end{equation}
and
\begin{equation}\label{eq: value eff GG}
\tag{R.6}
    \inf \; \Phi(\bm{\rho}_{\varepsilon}(T)),\quad \text{s.t. } \; \eqref{eq: granular 2}-\eqref{eq: granular 1}\,.
\end{equation}
The following corollary is a direct consequence of Theorem \ref{thm: conv G}.

\begin{cor}
    As $\varepsilon\to 0$, the problem \eqref{eq: value eps GG} is approximated by \eqref{eq: value eff GG} in the sense of Definition \ref{def:approx}. In particular, the controlled system \eqref{eq: granular relax} with any given initial conditions converges to \eqref{eq: granular 2}-\eqref{eq: granular 1} locally uniformly on any finite time interval, and we have $\inf\, \Phi(\bm{\rho}_{\varepsilon}(T)) \to \inf \, \Phi(\bm{\rho}(T))$. 
\end{cor}

\bibliographystyle{amsplain}
\bibliography{bibliography}

\end{document}